\documentclass[11pt]{article}
\usepackage{amsmath, amssymb}
\usepackage{graphicx}
\usepackage{geometry}
\usepackage{xcolor}
\usepackage{authblk} 
\usepackage{hyperref}
\usepackage{cleveref} 
\usepackage{amsthm}
\usepackage{enumitem} 
\usepackage{booktabs}
\usepackage{newtxtext,newtxmath}
\usepackage{setspace}
\usepackage{makecell}
\usepackage{subcaption}
\usepackage{multirow}
\usepackage{titlesec}
\usepackage{etoolbox}

\titleformat{\section}
{\normalfont\large\bfseries}
{\thesection}{1em}{}

\usepackage{caption}
\newtheorem{theorem}{Theorem}

\newtheorem{assumption}{Assumption}
 
\newtheorem{proposition}{Proposition}
\newtheorem{corollary}{Corollary} 
\newtheorem{example}{Example}

\renewcommand{\d}{\mathrm{d}}   

\hypersetup{colorlinks=true, allcolors=blue}  
\makeatletter
\patchcmd{\@maketitle}
  {\vskip 2em}
  {\vskip -4.5em}
  {}{}
\makeatother

\title{\Large \MakeUppercase{Pathwise skew-symmetric discretisation for SDEs with superlinear drift}}
\author[1]{Yuga Iguchi}
\author[2]{Samuel Livingstone}
\author[3]{Giorgos Vasdekis}
\author[1]{Rui-Yang Zhang}

\affil[1]{School of Mathematical Sciences, Lancaster University.}
\affil[2]{Department of Statistical Science, University College London.}
\affil[3]{School of Mathematics, Statistics and Physics, Newcastle University.}

\begin{document}
\maketitle 
\begin{abstract}
    The skew-symmetric discretisation has recently been proposed as a new robust simulation method for weakly approximating stochastic differential equations (SDEs) with non-globally Lipschitz drift. This work develops a pathwise version of the scheme by representing the noise increment as a skew-normal distribution and coupling it with the driving Brownian increments, thereby enabling its use in the multilevel Monte Carlo (MLMC) framework. Under suitable conditions, we establish strong convergence of order 1/2 in $L^2$. Subsequently, the associated MLMC estimator is shown to have computational complexity $\mathcal{O} \bigl(\varepsilon^{-2} (\log (1/\varepsilon))^2 \bigr)$ to achieve a mean-squared error $\varepsilon^2$. We then analytically compare the proposed scheme with the tamed Euler scheme, another benchmark for robust discretisation. Under a strong inward-drift regime with the current state being far from the stable region, we show that the probability of moving in the wrong direction tends to vanish in the skew-symmetric scheme, whereas the tamed Euler scheme makes such moves with a non-trivial probability. Furthermore, in the MLMC setting, employing a one-dimensional stochastic Ginzburg-Landau model, we specify the range of step sizes for which the asymptotic variance of the coupled level difference obtained via the pathwise skew-symmetric scheme is lower than that obtained via the tamed Euler scheme. Numerical experiments on several model examples support the theoretical rate of strong convergence and demonstrate the stability and effectiveness of the resulting MLMC in the superlinear drift setting.
\end{abstract}
\textbf{Key  words.} Stochastic differential equations; strong convergence; skew-normal distribution; unadjusted Barker algorithm; tamed Euler scheme; multilevel Monte Carlo.  


\section{Introduction} \label{sec:intro}

Simulating and estimating expectations of functionals of a stochastic differential equation (SDE), either at a fixed time or at equilibrium, is a basic task in molecular dynamics \cite{krauth2006statistical}, mathematical finance \cite{karatzas2014brownian} and computational statistics \cite{Roberts.Tweedie:96}. Expectations are typically approximated by simulating sample paths with an explicit scheme to form Monte Carlo averages. The classical Euler--Maruyama scheme is well understood when the coefficients are globally Lipschitz, but the scheme is also known to be unstable when the drift grows superlinearly, as moments may diverge in finite time, meaning the scheme can fail to be ergodic in the long-time regime (e.g. \cite{Hutzenthaler.Jentzen.Kloeden:2011, Roberts.Tweedie:96}). This has motivated a range of stable, explicit schemes, notably the tamed/truncated/adaptive Euler methods (e.g. \cite{Fang:20, hutzenthaler.jentzen.lloeden:12, mao:15, sabanis:13}).

A recent alternative is the skew-symmetric scheme of \cite{skew:ima}, in which the SDE transition kernel is approximated using a skew-symmetric distribution \cite{azzalini2013skew}. A skew-symmetric increment is simulated by first generating a symmetric Gaussian jump informed by only the volatility of the SDE, and then skewing it in the direction of the drift.  As the jump magnitude is independent of the drift, large values of the latter cannot destabilise the trajectory, meaning the scheme converges weakly with order one and remains ergodic under a one-sided Lipschitz condition. Applied to the overdamped Langevin dynamics, the skew-symmetric scheme recovers the unadjusted Barker sampling algorithm \cite{livingstone.zanella:22,skew:ima}.

\subsection{Contributions and structure of the paper}
Our contributions are as follows:
\begin{itemize}[leftmargin=0.5cm]
\item We give a pathwise formulation of the skew-symmetric scheme, coupling the increments with the SDE's Brownian path (Section \ref{sec:construction-pathwise-ss}). The resulting scheme is fully explicit and, aside from the choice of a fixed time grid, requires no additional tuning, such as the design of an adaptive step size or the selection of a truncation threshold. 
\item We show that the pathwise scheme attains a strong error in mean square of order $1/2$ in the superlinear drift setting (Section \ref{sec:strong.convergence}).
%
\item Building on this, we embed the scheme in the multilevel Monte Carlo (MLMC) framework (Section \ref{sec:multi.level.def}) and bound the cost of achieving a mean-square error $\varepsilon^2$ by $\mathcal{O} \bigl(\varepsilon^{-2} (\log (1 /  \varepsilon) )^2 \bigr)$.
%
\item We present a detailed comparative analysis with the tamed Euler scheme (Section \ref{sec:comparative.analysis}). In strongly inward large state regimes, we show that the probability of moving in the wrong direction approaches $0$ for the pathwise skew-symmetric scheme whereas it remains non-trivial for the tamed Euler scheme. For the stochastic Ginzburg-Landau model \eqref{eq:SGL}, we further characterise the range of step sizes for which the skew-symmetric scheme yields a smaller MLMC coupling variance.

\item We support our theoretical results through numerical simulations, demonstrating the strong convergence rate (Section \ref{sec:numerical.strong.convergence.rate}) and comparing the resulting MLMC estimator with that based on the tamed Euler scheme (Section \ref{sec:numerical.mlmc}). We observe that the pathwise skew-symmetric scheme has favourable robustness with respect to model and MLMC settings with multiplicative volatility and superlinear drift.

\end{itemize}

To motivate further analysis, we demonstrate the pathwise robustness of the scheme for the stochastic Ginzburg-Landau model with multiplicative noise, as defined in \eqref{eq:SGL}. The drift exhibits superlinear behaviour arising from the double-well potential $ U(x) = x^4/4 - \beta x^2/2 $. \cref{fig:intro_SGL} shows the trajectories generated by the pathwise skew-symmetric and tamed Euler schemes. 
While the true trajectories are expected to move downwards (towards the well at 1) due to the inward-pointing drift, we observe that some trajectories of the tamed Euler scheme fail to capture this correct downward trend at both step sizes.  In contrast, all the sample paths realised by the skew-symmetric scheme track the correct direction. A formal analysis for this observation is given in Section \ref{sec:inward_drift}, and more experiments are provided in Section \ref{sec:experiments}.
%


\begin{figure}
\centering
\vspace{-10pt}  
\includegraphics[width=0.8\linewidth]{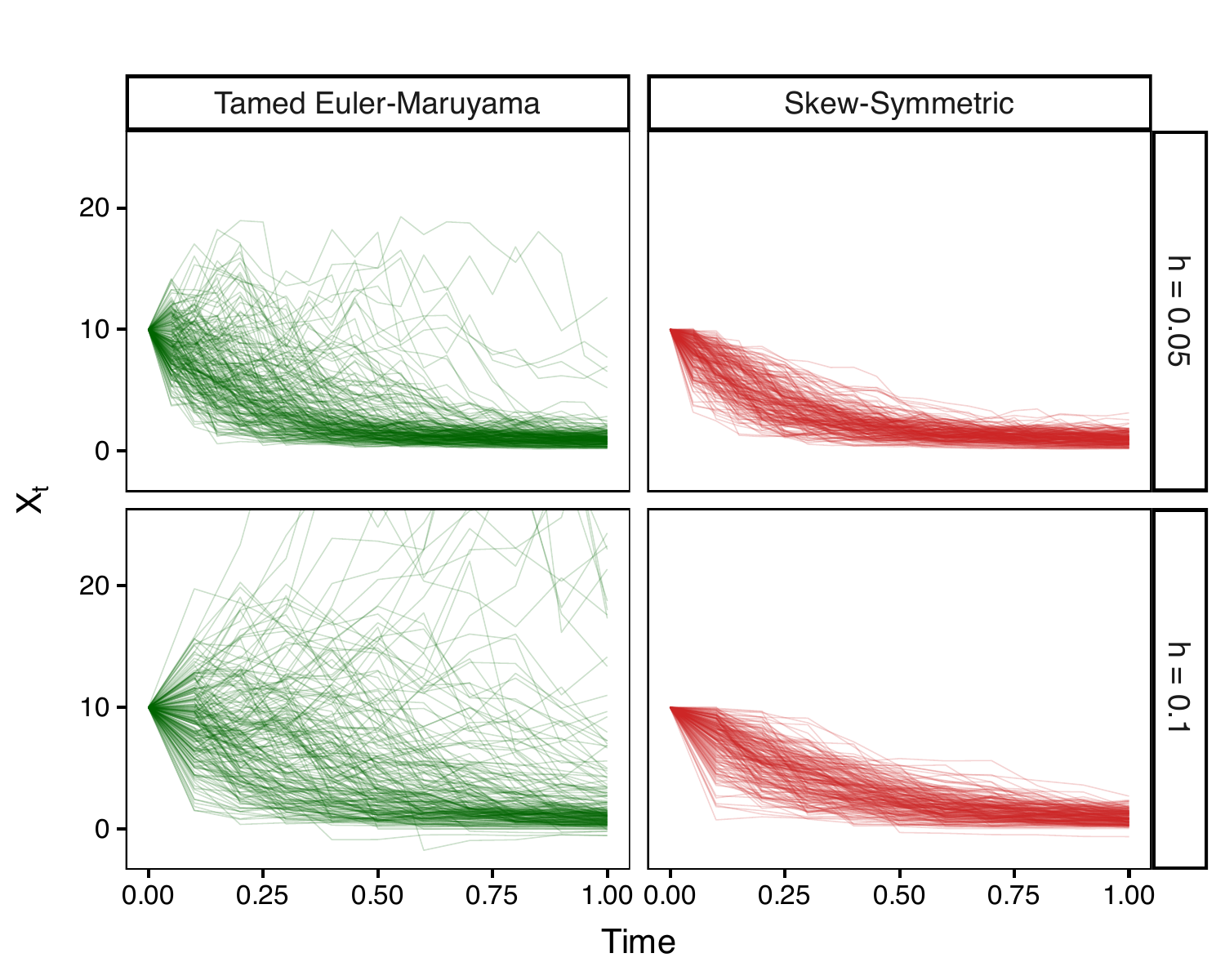}
\vspace{-8pt}   
\caption{200 trajectories of the stochastic Ginzburg-Landau model \eqref{eq:SGL} with $\beta = \gamma = 1$, generated by the tamed Euler (left) and pathwise skew-symmetric (right) schemes with step size $h = 0.05$ and $h = 0.1$, using the same underlying Brownian motion. The initial state is set as $X_0 = 10$. We observe that some tamed Euler trajectories are not moving in the correct direction (towards one), whereas such issues are not present for skew-symmetric trajectories.}
\vspace{-6pt}   
\label{fig:intro_SGL}
\end{figure}

\subsection{Set-up and notation}\label{sec:notation}
Our objective is to approximately simulate the $d$-dimensional stochastic differential equation (SDE)
\begin{align} \label{eq:sde}
\d X_t =  
V_{0} (X_t) \d t 
+ \sum_{i = 1}^{d_W} V_{i} (X_t)  \d W_{i, t}, \qquad  X_0 \in \mathbb{R}^d,
\end{align} 
where $W_t = (W_{1, t}, \ldots, W_{d_W, t}), \, t \ge 0$ is a $d_W$-dimensional standard Brownian motion and $V_{i}: \mathbb{R}^{d} \to \mathbb{R}^{d}, \ i = 0, 1, \ldots, d_W$. 
The matrix $V=[V_1,\ldots,V_{d_W}]$ is assumed to satisfy a uniform ellipticity condition, namely that $a(x):=V(x) \cdot V^T(x)$ is uniformly positive-definite in $x\in\mathbb R^d$.  
For the SDE (\ref{eq:sde}), we write the expectation under the law of $X$ starting from $x \in \mathbb{R}^d$ as $\mathbb{E}_x [\cdot]$. Also, $|\cdot|$ denotes the Euclidean norm on $\mathbb{R}^m$ for $m \in \mathbb{N}$, with the usual interpretation as absolute value for scalars.

Let $C^l(\mathbb{R}^d)$ denote the set of $l$-times differentiable functions
$f:\mathbb{R}^d \to \mathbb{R}$ with continuous derivatives up to order $l$.
We also denote
\[
C_P^{l,m}(\mathbb{R}^d)
=
\left\{
f \in C^l(\mathbb{R}^d)
:
\exists C>0 \ \text{s.t.} \ 
\lvert f^{(k)}(x)\rvert
\le C\bigl(1+|x|^m\bigr),
\ \forall x\in\mathbb{R}^d,\ k\le l
\right\},
\]
where $f^{(k)}$ denotes any derivative of order $k$, as well as
\[
C_P^l(\mathbb{R}^d)
=
\bigcup_{m\in\mathbb{N}} C_P^{l,m}(\mathbb{R}^d),
\qquad
C_P^\infty(\mathbb{R}^d)
=
\bigcap_{l\in\mathbb{N}} C_P^l(\mathbb{R}^d).
\]

\section{Pathwise skew-symmetric scheme}
\label{sec:main}
\subsection{Distributional skew-symmetric scheme -- review}\label{sec:review.ss}

The scheme of \cite{skew:ima} generates a solution to SDE \eqref{eq:sde} at fixed step size $h>0$ by replacing the additive drift increment of Euler--Maruyama with a probabilistic \emph{sign flip} applied to the Gaussian increment, where the flipping probability encodes the drift.
Given i.i.d. $\nu_n = (\nu_{1,n}, \dots, \nu_{d_W,n}) \sim N(0, I_{d_W})$, to propose the next step from $\bar{X}_n$ following the skew-symmetric scheme, one draws independent binary signs $b_{i,n} \in \{+1, -1\}$ with
\begin{equation}
  b_{i,n} =
  \begin{cases}
    +1 & \text{with probability } p_i(\bar{X}_n, \nu_n), \\
    -1 & \text{otherwise},
  \end{cases}
\end{equation}
with probabilities $p_i$ to ensure the dynamics are correct, and updates following
\begin{equation}
    \bar{X}_{n+1}
  = \bar{X}_n + h^{1/2}\sum_{i=1}^{d_W} b_{i,n}\, V_i(\bar{X}_n)\, \nu_{i,n}.
  \label{eq:ss-update}
\end{equation}
Geometrically, each column of $V$ is treated as a basis direction. A Gaussian jump is proposed along each, then either accepted or reflected according to the probability $p_i$. We impose a condition on the derivative of $p_i$ with respect to $\nu_i$ at $\nu_i=0$ to guarantee that the scheme is of weak-order 1. In practice, $p_i$ can be constructed as 
\begin{equation*}
  p_i(x, \nu) = F\left( \nu_i\, h^{1/2}\, \frac{\alpha_i(x)}{2 f(0)} \right),
\end{equation*}
with $f$ and $F$ being the density and cumulative distribution function (CDF) of any centred and symmetric random variable on $\mathbb{R}$ with $f(0) \neq 0$, and
\begin{equation}\label{def:alpha}
    \alpha(x) = (\alpha_1(x) , \dots , \alpha_{d_W}(x)) := V_0^T(x) \cdot a^{-1}(x) \cdot V(x).
\end{equation}
While many symmetric distributions will result in a weak-order 1 scheme, choosing the Gaussian density and CDF will allow us to couple the jump of the scheme with the Brownian motion driving the SDE in a natural way \cite{Hen:86}. This coupling leads to a strong convergence result. In this article, we thus take $p_i$ to be of the form
\begin{equation} \label{eq:p-construction:Gaussian}
  p_i(x, \nu) = \Phi\!\left( \nu_i\, h^{1/2}\, \frac{\alpha_i(x)}{2 \phi(0)} \right),
\end{equation}
with $\Phi$ and $\phi$ the standard Gaussian CDF and density on $\mathbb{R}$.  

\subsection{Construction of the pathwise skew-symmetric scheme} \label{sec:construction-pathwise-ss}

By selecting the Gaussian densities for the general skew-symmetric updates of \eqref{eq:ss-update}, we can now construct a pathwise skew-symmetric scheme. The key ingredient of our proposed pathwise scheme is that a skew-normal distribution (e.g. $b_{i,n} \nu_{i,n}$) can be rewritten using two independent Gaussians \cite{Hen:86}, which enable successful couplings.

Here, the pathwise scheme involves the Brownian motion $\{W_t \}_{t \ge 0}$ and another independent $d_W$-dimensional Brownian motion $\{ \widetilde{W}_t \}_{t \ge 0}$, instead of the random variables $b_{n} = (b_{1, n}, \ldots, b_{d_W, n}) \in \{-1,+1\}^{d_W}$ used in the distributional sense for \eqref{eq:ss-update}. 

To begin, note that for $\nu_n = (\nu_{1, n}, \ldots, \nu_{d_W, n}) \sim N(0,I_{d_W})$ and $(h, x) \in (0, \infty) \times \mathbb{R}^d$, the random variables $\xi_{i, n} \equiv b_{i, n}  \times \nu_{i, n}, \, i = 1, 2, \ldots, d_W$ under the Gaussian flipping probability \eqref{eq:p-construction:Gaussian} admit the probability density
\begin{align}
2 \phi (\xi_{i, n}) \Phi (\tilde{\alpha}_i (h, x) \xi_{i, n}) 
\end{align} 
where we have set
\begin{align} \label{eq:skewness}
\tilde{\alpha}_i (h,x) = \frac{h^{1/2}}{2 \phi(0)} \alpha_i (x), 
\end{align} 
with $\alpha_i$ as in (\ref{def:alpha}); see e.g. \cite{Hen:86} for a derivation.
%
%
This means that the random variable $\xi_{i, n}$ has a skew-normal distribution with skew parameter $\tilde{\alpha}_i (h,x)$. It is then shown in \cite{Hen:86} that $\xi_{i, n}$ can be represented as 
\begin{align}
\xi_{i, n} := \frac{1}{\sqrt{1 + \tilde{\alpha}_{i} (h,x)^2}} \nu_{i, n} 
+ \frac{\tilde{\alpha}_i (h, x)}{\sqrt{1 + \tilde{\alpha}_i (h,x)^2}} |\eta_{i,n}|,  
\end{align} 
with $\eta_{i,n} \sim N(0,1)$ independent of $\nu_{i,n}$. Using the Brownian motions $\{(W_t, \widetilde{W}_t)\}_{t \ge 0}$, we therefore define the pathwise skew-symmetric scheme with a fixed step size $h > 0$ and initial condition $\widetilde X_0=x$ by
%
\begin{align} \label{eq:pathwise_ss}
\widetilde{X}_{n+1} = \widetilde{X}_{n} + \sum_{i = 1}^{d_W} 
{V}_{i} (\widetilde{X}_{n}) 
\left( \frac{\Delta W_{i, n} + \tilde{\alpha}_i (h, \widetilde{X}_{n})| \Delta \widetilde{W}_{i, n} | }{\sqrt{1 + \tilde{\alpha}_i (h, \widetilde{X}_{n})^2}}  \right)
\end{align}  
for $n \in \mathbb{N} \cup \{0\}$, 
where $\Delta{W}_{i, n} = W_{i, (n+1)h} - W_{i, nh}$ and $\Delta \widetilde{W}_{i, n} = \widetilde{W}_{i, (n+1)h} - \widetilde{W}_{i, nh}$. 
%


\subsection{Strong error analysis}\label{sec:strong.convergence}
For the proposed scheme \eqref{eq:pathwise_ss}, we now state the mean-square convergence, particularly under the setting of superlinear drift. We follow the argument of \cite{Tret:13} and first establish the local errors, which lead to the rate of mean-square error $1/2$. We impose the following assumptions on the SDE (\ref{eq:sde}). 

\begin{assumption}[Drift] \label{assump:drift}
There exist constants $C_1, C_2 > 0$ and $q \ge 1$ such that for all $x, y \in \mathbb{R}^d$, the drift $V_{0}$ satisfies
\begin{align}
& \bigl( x - y \bigr)^\top \bigl( V_0 (x) - V_0 (y) \bigr) 
\le C_1 |x - y|^2; \\[0.2cm] 
& | V_{0} (x) - V_{0} (y) |^2  
\le C_2 \bigl(1 + |x|^{2q} + |y|^{2q} \bigr) | x - y |^2. 
\end{align} 
\end{assumption}

\begin{assumption}[Volatility] \label{assump:vol}
\begin{itemize}
\item[I.] The matrix $a (x) \equiv V (x) V^\top (x), \, V = [V_1, \ldots, V_{d_W}]$ is positive-definite uniformly in $x \in \mathbb{R}^d$. 
\item[II.] There exists a constant $C > 0$ such that for all $x, y \in \mathbb{R}^d$,   
\begin{align}
\sum_{i = 1}^{d_W} | V_{i} (x) 
- V_{i} (y) |^2 \le C | x - y|^2. 
\end{align}
\end{itemize}
\end{assumption}

\begin{assumption}[Regularity] \label{assump:deriv_poly}
For any $(i, j) \in  \{0, 1, \ldots, d_W \} \times \{1, \ldots, d\}$, the $j$th component $V_{i, j}$ of $V_i$ is in $C_P^\infty (\mathbb{R}^d)$. 
%
\end{assumption} 

\begin{proposition}[Local error control] \label{prop:one_step_s}
Let Assumptions \ref{assump:drift}, \ref{assump:vol} and \ref{assump:deriv_poly} hold. For any $r \ge 1$, there are constants $C > 0$, $q \ge 1$ and $h_0 > 0$ such that for all $h \le  h_0$ and $x \in \mathbb{R}^d$, 
\begin{gather}
\bigl|\mathbb{E}_x  \bigl[ X_h \bigr] - \mathbb{E}_x  \bigl[ \widetilde{X}_1  \bigr]  \bigr| 
\le C ( 1 + |x|^{2q} )^{1/2} \times h^2;  \label{eq:local_mean} \\[0.2cm]
\Bigl\{  \mathbb{E}_x \bigl[ | X_h - \widetilde{X}_1|^{2r}  \bigr] 
\Bigr\}^{{1}/{2r}} 
\le  
C(  1 +  | x |^{2 rq}  )^{{1}/{2r}}  \times h. \label{eq:local_Lp}
\end{gather}
\end{proposition}

The proof of \cref{prop:one_step_s} is postponed to Appendix \ref{sec:pf_prop}. 
\\ 

Using \cref{prop:one_step_s}, the main result of this section establishes an upper bound on the mean-square convergence of the pathwise skew-symmetric scheme (\ref{eq:pathwise_ss}) below.
\begin{theorem}[Strong convergence in mean square] \label{thm:main_strong}
Let Assumptions \ref{assump:drift}, \ref{assump:vol} and
\ref{assump:deriv_poly} hold, and fix $T>0$. Then, for any
$r\geq 1$, there exist constants $C>0$ and $q\geq 1$ such that,
for all $h\in(0,1)$, $x\in\mathbb{R}^d$, and
$k\in\mathbb{N} \cup\{0\}$ satisfying $kh\leq T$, the pathwise
skew-symmetric scheme defined in (\ref{eq:pathwise_ss}) satisfies
\begin{align}
\Bigl\{  \mathbb{E}_x \bigl[  |  X_{k h} - \widetilde{X}_{k} |^{2r} \bigr] \Bigr\}^{{1}/{2r}} \le C ( 1 +  |x|^{2rq})^{{1}/{2r}} \times h^{{1}/{2}}
\end{align}
where $C$ and $q$ are independent of $h$, $k$, and $x$, but may
depend on $T$ and $r$. The order of root-mean-square convergence is therefore ${1}/{2}$. 
\end{theorem} 
\begin{proof}[Proof of Theorem \ref{thm:main_strong}]
We check the sufficient conditions in \cite[Theorem 2.1]{Tret:13} on the fixed time horizon $[0,T]$. The first and second conditions therein hold respectively from our Assumptions \ref{assump:drift}-\ref{assump:vol} and \cref{prop:one_step_s}.  
Finally, the third condition (finite moments of the scheme) follows from \cite[Lemma C.2]{skew:ima} by noticing that the pathwise scheme $\widetilde{X}_n$ and $\bar{X}_n$ defined in \eqref{eq:ss-update} share the same distribution. 
\end{proof} 

\subsection{Multi-level extension}\label{sec:multi.level.def} 
The pathwise construction of the skew-symmetric scheme in \eqref{eq:pathwise_ss} naturally facilitates its use within the \emph{multilevel Monte Carlo} framework \cite{Giles:08, Hut:13}. We consider the quantity $\mathbb{E}_x [f (X_T)]$ for a fixed time $T >0$, a starting state $x \in \mathbb{R}^d$ and a test function $f: \mathbb{R}^d \to \mathbb{R}$ as the target and assess the computational complexity of the multi-level estimator based on the proposed scheme. We define the step sizes $h_\ell = T / 2^{\ell}, \, \ell = 0, 1, \ldots, L$ for an integer $L > 0$ and write 
$K_\ell = 2^{\ell}$. Denote by $\widetilde{X}_{K_\ell}^{[i]}$ the $i$th independent realisation of the scheme (\ref{eq:pathwise_ss}) at time $T$ with the step size $h_\ell$. Then, the hierarchical estimators $Y_\ell, \ell = 0, 1, \ldots, L$ are defined as  
\begin{align}
Y_0 = \tfrac{1}{N_0} \sum_{i = 1}^{N_0} f (\widetilde{X}_{K_0}^{[i]}), 
\quad  
Y_\ell = \tfrac{1}{N_\ell} \sum_{i = 1}^{N_\ell} \Bigl[ 
f (\widetilde{X}_{K_{\ell}}^{[i]})
- f (\widetilde{X}_{K_{\ell - 1}}^{[i]}) 
\Bigr], \quad  1 \le \ell \le L, 
\end{align}
where $N_\ell$ is the number of Monte Carlo paths assigned to each estimator $Y_\ell$. For each $\ell \geq 1$, the fine and coarse paths are coupled using the same underlying Brownian motions, with coarse increments defined by $\Delta W_{i,n}^{(\ell-1)}=\Delta W_{i,2n}^{(\ell)}+\Delta W_{i,2n+1}^{(\ell)}$ and $\Delta\widetilde W_{i,n}^{(\ell-1)}=\Delta\widetilde W_{i,2n}^{(\ell)}+\Delta\widetilde W_{i,2n+1}^{(\ell)}$, so that the coarse update uses $\bigl|\Delta\widetilde W_{i,n}^{(\ell-1)}\bigr|$.  Finally, the multilevel estimator is defined as $Y = \sum_{\ell=0}^L Y_{\ell}$. 

The following result quantifies the computational cost $C$ required to control the mean-square error of the multilevel estimator. Here $C$ denotes the total number of drift and diffusion evaluations required by the algorithm.

\begin{corollary}[Computational complexity of Multi-Level estimator] \label{cor:mlmc}
Fix $T >0$ and let $f \in C_P^\infty(\mathbb{R}^d)$.  
Let Assumptions \ref{assump:drift}, \ref{assump:vol} and \ref{assump:deriv_poly} hold. Then there exists a constant $c > 0$ such that for any $\varepsilon < e^{-1}$, there are values $L$ and $N_\ell$ satisfying 
\begin{align}
\mathbb{E}_x \bigl[ \bigl( Y - \mathbb{E}_x [f (X_T)] \bigr)^2 \bigr] < \varepsilon^2, 
\end{align}
with a computational complexity
\begin{align} \label{eq:cost_bd}
C \le c \, \varepsilon^{-2} ( \log  (1/\varepsilon) )^2.  
\end{align}
\end{corollary}

\begin{proof}[Proof of Corollary \ref{cor:mlmc}]
We verify the four sufficient conditions of \cite[Theorem 3.1]{Giles:08} with $\alpha = \beta = 1$.

Condition (i) requires order $\alpha$ weak convergence of the numerical scheme. By \cite[Theorem 3.2]{skew:ima}, this holds with $\alpha = 1$ under Assumptions \ref{assump:drift}, \ref{assump:vol} and \ref{assump:deriv_poly}. Condition (ii) is trivially satisfied as the estimators within each level are simple Monte Carlo and use the linearity of expectation.  

For condition (iii), since $Y_\ell$ is the mean of $N_\ell$ independent coupled pairs,
$$\mathrm{Var}_x[Y_\ell] \leq \frac{1}{N_\ell} \mathbb{E}_x \bigl[|f(\widetilde{X}_{K_\ell}) - f(\widetilde{X}_{K_{\ell-1}})|^2\bigr].$$
Since $f \in C_P^\infty(\mathbb{R}^d)$, it can be shown using the mean value theorem that $f$ satisfies the polynomial Lipschitz condition
\[
|f(x) - f(y)|^2 \leq \widetilde{C}(1+|x|^q + |y|^q)|x-y|^2
\] 
for some $\widetilde{C} < \infty$, $q > 0$ and all $x,y\in\mathbb{R}^d$.  Using this and the Cauchy--Schwartz inequality gives
\begin{align*}
& \mathbb{E}_x \bigl[|f(\widetilde{X}_{K_\ell}) - f(\widetilde{X}_{K_{\ell-1}})|^2\bigr] 
\leq 
\widetilde{C}~\mathbb{E}_x \Bigl[\bigl(1 + |\widetilde{X}_{K_\ell}|^q + |\widetilde{X}_{K_{\ell-1}}|^q\bigr)|\widetilde{X}_{K_\ell} - \widetilde{X}_{K_{\ell-1}}|^2\Bigr]
\\
&\qquad \leq \widetilde{C}\sqrt{\mathbb{E}_x\Bigl[\bigl(1 + |\widetilde{X}_{K_\ell}|^q + |\widetilde{X}_{K_{\ell-1}}|^q\bigr)^2\Bigr]} \cdot \sqrt{\mathbb{E}_x\bigl[|\widetilde{X}_{K_\ell} - \widetilde{X}_{K_{\ell-1}}|^4\bigr]}.
\end{align*}
The first factor is bounded uniformly in $\ell$ using the moment bounds in Theorem \ref{thm:main_strong}. For the second, the triangle inequality in $L^4$ and the $L^4$ strong convergence from Theorem \ref{thm:main_strong} give
\begin{align*}
\mathbb{E}_x\bigl[|\widetilde{X}_{K_\ell} - \widetilde{X}_{K_{\ell-1}}|^4\bigr]^{1/4} 
\leq 
\mathbb{E}_x\bigl[|\widetilde{X}_{K_\ell} - X_T|^4\bigr]^{1/4} + \mathbb{E}_x\bigl[|\widetilde{X}_{K_{\ell-1}} - X_T|^4\bigr]^{1/4} 
\leq 
c h_\ell^{1/2}
\end{align*}
for some $c > 0$, meaning $\mathbb{E}_x[|\widetilde{X}_{K_\ell} - \widetilde{X}_{K_{\ell-1}}|^4] \leq c^4 h_\ell^2$. Combining gives
$
\mathrm{Var}_x[Y_\ell] \leq c_2 h_\ell/N_\ell,
$
which establishes condition (iii) with $\beta = 1$.

Finally, for $\ell \geq 1$, each coupled sample requires simulating one fine path with $K_\ell = 2^\ell$ steps and one coarse path with $K_{\ell-1} = 2^{\ell-1}$ steps, giving a per-sample cost of $O(2^\ell) = O(h_\ell^{-1})$. Hence $C_\ell$, the computational complexity of $Y_\ell$, satisfies $C_\ell \leq c_3 N_\ell h_\ell^{-1}$, confirming condition (iv).
\end{proof}



\section{Comparative analysis}\label{sec:comparative.analysis}

\subsection{Directional stability in the large state regime} \label{sec:inward_drift}

The following comparison focuses on a strong inward-drift regime, where the states in several coordinates are far from the stable region and the drift points strongly inward relative to the local volatility. This regime is typical for stochastic Ginzburg--Landau equations and overdamped Langevin dynamics with superquadratic potentials. Throughout this subsection, we consider the $d$-dimensional stochastic differential equation driven by $d$-dimensional standard Brownian motion 
\begin{align} \label{eq:sde_case_study}
\d X_t = b (X_t) \d t + \Sigma (X_t) \d W_t, \quad X_0 = x \in \mathbb{R}^d, 
\end{align}
where $b : \mathbb{R}^d \to \mathbb{R}^d$ and diagonal $\Sigma (x) = \mathrm{diag} \bigl[ \sigma_1 (x), \ldots, \sigma_d (x) \bigr]$ with $\sigma_i : \mathbb{R}^d \to \mathbb{R}$. Equation (\ref{eq:sde_case_study}) then admits the component-wise expression
\begin{align} 
\d X_{i,t} = b_{i} (X_{t}) \d t + \sigma_i  (X_t) \d W_{i, t}, \quad i = 1, \ldots, d.   
\end{align}
We study the probability that the one-step skew-symmetric scheme (\ref{eq:pathwise_ss}) captures the wrong direction when the drift function $b_i$ is strongly directed inwards. In particular, we focus on the scenario where some coordinates of the current state $x$ are far from the origin, and therefore the drift $b_i$ is such that for small time steps the next state will move toward the origin with probability close to one. 
For a non-empty set $\Lambda \subseteq \{1, \ldots, d\}$, we define $\rho_ {\Lambda} (x) := \min_{i \in \Lambda} |x_i|$, and limits  $\rho_{\Lambda} (x) \to  \infty$ are understood with the remaining coordinates staying in a fixed bounded set. We then assume that there exists a non-empty set $\Lambda \subseteq \{1, \ldots, d\}$ such that 
\begin{align} 
\label{eq:drift_inward}
\displaystyle 
\lim_{\rho_{\Lambda} (x) \to \infty} \max_{i \in \Lambda} \, 
\mathrm{sign} (x_i) b_i (x) = - \infty.  
\end{align}  
We also assume throughout this subsection that $\sigma_i(x)\neq0$ for every $i\in\Lambda$ and every state $x$ under consideration.  For a fixed step size $h > 0$, a single step of the pathwise skew-symmetric scheme gives
\begin{align} \label{eq:ss_one_step}
\widetilde{X}_{i, 1}
= x_i +  \sigma_i (x) \sqrt{h} \left( \frac{\xi_i + \tilde{\alpha}_i (h, x) |\eta_i|}{\sqrt{1 + \tilde{\alpha}_i (h, x)^2}} 
\right)
\end{align}
for $i \in \{1,...,d\}$, where $\xi = (\xi_1, \ldots, \xi_d), \, \eta = (\eta_1, \ldots, \eta_d) \sim N(0, I_d)$ are independent and
\[
\tilde{\alpha}_i (h, x) = \sqrt{\frac{\pi}{2}} \frac{b_i (x)}{\sigma_i (x)} \sqrt{h}. 
\]
For any $\varepsilon \ge 0$ and $x \in \mathbb{R}^d$, we define the event under which the scheme captures wrong directions as
\begin{align*}
A_{i} (\varepsilon, x) := \Bigl\{  \mathrm{sign} (x_i) \bigl( \widetilde{X}_{i, 1} -  x_i \bigr) > \varepsilon \Bigr\}
= \begin{cases}  
\bigl\{ 
\widetilde{X}_{i, 1} >  x_i  + \varepsilon \bigr\},  & x_i > 0; \\[0.2cm] 
\bigl\{ 
\widetilde{X}_{i, 1} <  x_i  - \varepsilon   \bigr\} & x_i < 0,   
\end{cases} \quad i \in \Lambda.   
\end{align*}
We also consider the component-wise tamed Euler scheme, whose one-step transition is given by
\begin{align} \label{eq:te_one_step}
\hat{X}_{i, 1}
= x_i + \frac{h b_i (x)}{1 + h | b_i (x)|} + \sigma_i (x) \sqrt{h} \xi_i
\end{align} 
for $i \in \{1,...,d\}$, where $\xi \sim N(0, I_d)$, and its wrong direction event is written as 
\begin{align}
B_{i} (\varepsilon, x) := \Bigl\{ 
\mathrm{sign} (x_i) \bigl( \hat{X}_{i, 1} -  x_i \bigr) > \varepsilon \Bigr\}, \qquad \varepsilon \ge 0, \ x \in \mathbb{R}^d,  \ i \in \Lambda.
\end{align}    

%

The proposition below illustrates that the tamed Euler scheme will often fail to produce a path that travels in the correct direction, but that the skew-symmetric scheme does not suffer the same issue.

\begin{proposition} \label{prop:large_state}
Let $\varepsilon \ge 0$, $d \in \mathbb{N}$. Also, let $h > 0$ be a fixed step size. Consider the SDE \eqref{eq:sde_case_study} with its drift satisfying \eqref{eq:drift_inward} for some non-empty  $\Lambda \subseteq \{1, \ldots, d\}$. Then, we have the following.
\begin{enumerate}[leftmargin=0.5cm]
\item[(i)] Tamed Euler scheme (\ref{eq:te_one_step}). It holds for $i \in \Lambda$ that
\begin{align}
\mathbb{P} \Bigl( B_i (\varepsilon, x)  \Bigr)  = 1 - \Phi \left( \frac{1}{\sqrt{h} |\sigma_i (x)|} \left( \varepsilon - \frac{h b_i(x)}{1 + h | b_i (x)|} \mathrm{sign} (x_i) \right) \right),
\end{align}
where $z \mapsto \Phi (z)$ denotes the standard Gaussian CDF on $\mathbb{R}$. 
If furthermore it holds that $\lim_{\rho_{\Lambda}(x) \to \infty} |\sigma_i (x)| = \widetilde{\sigma}_i \in (0, \infty]$, then (\ref{eq:drift_inward}) implies
\begin{align} \label{eq:te_limit}
\lim_{\rho_\Lambda (x) \to \infty} \mathbb{P} \left( \bigcup_{i \in \Lambda} B_i (\varepsilon, x) \right) = 1 - \prod_{i \in \Lambda} \Phi \left( \tfrac{1}{\sqrt{h} \widetilde{\sigma}_i} ( \varepsilon + 1) \right). 
\end{align}
\item[(ii)] Skew-symmetric scheme  (\ref{eq:ss_one_step}). It holds for $i \in \Lambda$ that
\begin{align}
\begin{aligned} 
& \mathbb{P} \bigl( A_i (\varepsilon, x) \bigr) \nonumber \\
& \quad = 1 - \mathbb{E}_{\eta_i \sim N(0,1)} 
\left[ \Phi \left(\frac{\sqrt{1 + \tilde{\alpha}_i (h, x)^2}}{\sqrt{h} |\sigma_i (x)|} \varepsilon - \sqrt{\frac{\pi}{2}} \frac{b_i (x)}{|\sigma_i (x)|} \sqrt{h} | \eta_i | \mathrm{sign} (x_i) \right) \right].  
\end{aligned} 
\end{align}
Furthermore, if it holds $\lim_{\rho_\Lambda (x) \to \infty} \mathrm{sign} (x_i) b_i (x) / | \sigma_i (x) | = - \infty$ for all $i \in \Lambda$, then 
\begin{align} \label{eq:ss_limit}
\lim_{\rho_{\Lambda} (x) \to \infty} \mathbb{P} \left( \bigcup_{i \in \Lambda} A_i (\varepsilon, x) \right) = 0.  
\end{align}
\end{enumerate}
\end{proposition}

The proof is postponed to Appendix \ref{sec:pf_large_state}. We illustrate the key implications from the above result through the following two model examples in applications.

\begin{example} \label{eg:ginzburg}
Consider the one-dimensional Stochastic Ginzburg-Landau model defined as
\begin{align} \label{eq:SGL}
\d X_t = (- X_t^3 + \beta X_t) \d t + \gamma X_t \d W_t, 
\end{align}
with parameters $\beta, \gamma > 0$. We see that the drift exhibits a strong inward property satisfying $\mathrm{sign} (x) b (x) \to - \infty$ where $b(x) := -x^3 + \beta x$. The limit conditions stated in Proposition \ref{prop:large_state} also hold, in particular $\mathrm{sign}(x) b (x)/ |\sigma (x)| \to - \infty$ and $\lim_{x \to \pm \infty} | \sigma (x) | \to \infty$.  

The probability that the next state of the skew-symmetric scheme captures the correct direction, as equation \eqref{eq:ss_limit} holds, converges to $1$ as the current state $x \to \pm \infty$. 

In contrast, the tamed Euler scheme moves in the wrong direction with a limiting probability of $1/2$ as \eqref{eq:te_limit} implies for any fixed $\varepsilon \ge 0$ and step size $h > 0$,  
\begin{align*}
\lim_{x \to \pm \infty} \mathbb{P} \Bigl( B_1  (\varepsilon, x) \Bigr)
= 1 - \Phi \bigl( 0 \bigr) = \tfrac{1}{2}.
\end{align*}
This agrees with the definition of the scheme (\ref{eq:te_one_step}), that is, due to the large drift and volatility, the noise term $\sigma_i (x) \sqrt{h} \xi_i$ dominates the tamed drift, which becomes approximately $- \mathrm{sign}(x)\times1$ when $|x|$ is large.  Consequently, almost half of the trajectories of the one-step tamed Euler scheme fail to capture the correct direction, as demonstrated in \cref{fig:intro_SGL} of Section \ref{sec:intro}. 
\end{example}
\begin{example}
We now consider the $d$-dimensional \emph{overdamped Langevin equation}
\begin{align*}
\d X_t = - \nabla U (X_t) \d t + \sqrt{2} \d W_t,  
\end{align*}
with the potential $U$ of the form $U(x) = \sum_{i = 1}^d U_i (x_i), \, x \in \mathbb{R}^d$ with $U_i (z) = z^4/4 - \beta z^2/2, \, z \in \mathbb{R}$ for a parameter $\beta > 0$. We observe that the negative gradient of the potential exhibits a strong inward drift in all coordinates, meaning (\ref{eq:drift_inward}) holds with $\Lambda = \{1, \ldots, d\}$. 

Similar to the previous example, the second statement in \cref{prop:large_state} suggests that the skew-symmetric scheme moves in the wrong direction with a probability that converges to zero when the current state $x \in \mathbb{R}^d$ is large. 

On the other hand, for the tamed Euler scheme, we have
\begin{align*}
\lim_{\rho_\Lambda (x) \to \infty} \mathbb{P} \left( \bigcup_{i = 1}^d B_i  (\varepsilon, x) \right) = 1 - \left\{ \Phi \left( \frac{1}{\sqrt{2 h}} (\varepsilon + 1) \right) \right\}^d \to 1, \quad d \to \infty, 
\end{align*} 
so the probability that at least one of the coordinates of the tamed Euler scheme moves in the wrong direction is close to one for large state $x$ and dimension $d$. 
\end{example}

\subsection{Dynamics comparisons for stochastic Ginzburg-Landau} 

We further compare the two schemes against the true dynamics of the one-dimensional Stochastic Ginzburg-Landau model \eqref{eq:SGL} to assess how well they capture the exact dynamics. 

Let $h > 0$ be a fixed step size, $\varepsilon \ge 0$ and $\Delta_h^x = Y_h^x - x$, where $Y_h^x$ denotes the exact solution at time $h$ starting from $x$ or its corresponding numerical approximation. We then compare those dynamics via the following three criteria:
\begin{enumerate}
\item Wrong-direction probability, i.e., $\lim_{x \to \infty} \mathbb{P} \bigl( \Delta_h^x  > \varepsilon \bigr)$; 
\item Normalised mean increment, i.e., $\lim_{x \to \infty} \mathbb{E}_x [\Delta_h^x] / x$;
\item Normalised $L^2$ increment, i.e., $\lim_{x \to \infty} \sqrt{\mathbb{E}_x [(\Delta_h^x)^2]}/x$.
\end{enumerate}
Table \ref{tab:one_step_signed_mean} summarises the three asymptotic quantities for each dynamics. The limits are obtained by elementary calculus. For the true dynamics, we used the closed-form expression of the solution \[
X_h = \exp \bigl( (\beta - \gamma^2/2) h + \gamma B_h \bigr) / \sqrt{x^{-2} + 2 \int_0^h \exp \big( (2 \beta - \gamma^2) s + 2 \gamma B_s \bigr) ds }
\]
which can be obtained from the It\^o's formula for $Z_t = X_t^{-2}$. The table reveals that the skew-symmetric scheme reproduces the qualitative behaviour of the exact dynamics under large-state scaling. In contrast, the tamed Euler scheme exhibits the random-walk behaviour as discussed before, and its normalised mean increment vanishes in the large-state regime. 

\begin{table}[h]
\centering
\caption{Comparison of exact dynamics and its one-step numerical approximations in the large-state regime \(x\to+\infty\) for the one-dimensional Stochastic Ginzburg-Landau model \eqref{eq:SGL} with a fixed step size $h > 0$ and $\varepsilon \ge  0$. } 
\label{tab:one_step_signed_mean}
\renewcommand{\arraystretch}{1.25}
\setlength{\tabcolsep}{12pt}
\begin{tabular}{@{}lccc@{}}
\toprule
Dynamics
& \makecell{\(\displaystyle \lim_{x \to \infty} \mathbb{P} \bigl( \Delta_h^x > \varepsilon \bigr)\)} 
& \makecell{\(\displaystyle \lim_{x \to \infty} \frac{1}{x} {\mathbb{E} \bigl[ \Delta_h^x  \bigr]}\)}
& \makecell{\(\displaystyle \lim_{x \to \infty} \frac{1}{x} \sqrt{\mathbb{E} \bigl[ (\Delta_h^x)^2  \bigr]} \)} \\[0.2cm] 
\midrule
Exact solution
&
\(\displaystyle 0 \)
&
\(\displaystyle -1 \)
&
\(\displaystyle 1 \)
\\[0.2cm]
\makecell[l]{Tamed Euler \\ scheme}
&
\(\displaystyle \frac{1}{2} \)
&
\(\displaystyle 0 \)
&
\(\displaystyle \gamma \sqrt{h} \)
\\[0.2cm]
\makecell[l]{Skew-symmetric\\ scheme}
&
\(\displaystyle 0 
\)
&
\(\displaystyle 
- \gamma \sqrt h \sqrt{\frac{2}{\pi}} 
\)
&
\(\displaystyle \gamma \sqrt{h} \)
\\
\bottomrule
\end{tabular}
\end{table}

\subsection{Variance analysis under the multi-level setting} \label{sec:var.multilevel}

Consider again the one-dimensional stochastic Ginzburg--Landau model defined in \eqref{eq:SGL}.
We fix a step size $h>0$ and study the multilevel Monte Carlo (MLMC) coupling between one coarse step of size $h$ and two fine steps of size $h/2$, in the regime $x\to\infty$.  Let $(U_1,U_2,V_1,V_2,G_1,G_2)$ be independent standard Gaussian random variables, and set
\[
U_c := \frac{U_1+U_2}{\sqrt{2}},
\qquad
V_c := \frac{V_1+V_2}{\sqrt{2}},
\qquad
G_c := \frac{G_1+G_2}{\sqrt{2}}.
\]
We also set $a(h) := \gamma \sqrt{h/2}$, and when convenient write $b(x) = -x^3 + \beta x$. We study the following numerical schemes:

\paragraph{(i) Tamed Euler}
For $\eta>0$, set
\begin{equation}
\label{eq:tamed}
T_\eta(y;G)
:=
y + \frac{\eta\,b(y)}{1+\eta |b(y)|} + \gamma y \sqrt{\eta}\,G.
\end{equation}
Define the coupled coarse and fine updates by
\[
Y_h^c(x) := T_h(x;G_c),
\qquad
Y_{h/2}^f(x) := T_{h/2}\bigl(T_{h/2}(x;G_1);G_2\bigr).
\]

\paragraph{(ii) Skew-symmetric scheme}
For $\eta>0$, define
\[
\tilde{\alpha}(\eta, y) := \sqrt{\frac{\pi}{2}} \sqrt{\eta}\,\frac{b(y)}{\gamma y},
\qquad
\beta_\eta(y) := \frac{1}{\sqrt{1+\tilde{\alpha}(\eta, y)^2}},
\qquad
\delta_\eta(y) := \frac{\tilde{\alpha}(\eta,y)}{\sqrt{1+\tilde{\alpha}(\eta, y)^2}},
\]
and set
\begin{equation}
\label{eq:skew}
S_\eta(y;U,V)
:=
y + \gamma y \sqrt{\eta}\bigl(\beta_\eta(y)V + \delta_\eta(y)|U|\bigr).
\end{equation}
Define the coupled coarse and fine updates in this case by
\[
\tilde{X}_h^c(x) := S_h(x;U_c,V_c),
\qquad
\tilde{X}_{h/2}^f(x) := S_{h/2}\bigl(S_{h/2}(x;U_1,V_1);U_2,V_2\bigr).
\]
\begin{figure}
    \centering
    \includegraphics[width=0.8\linewidth]{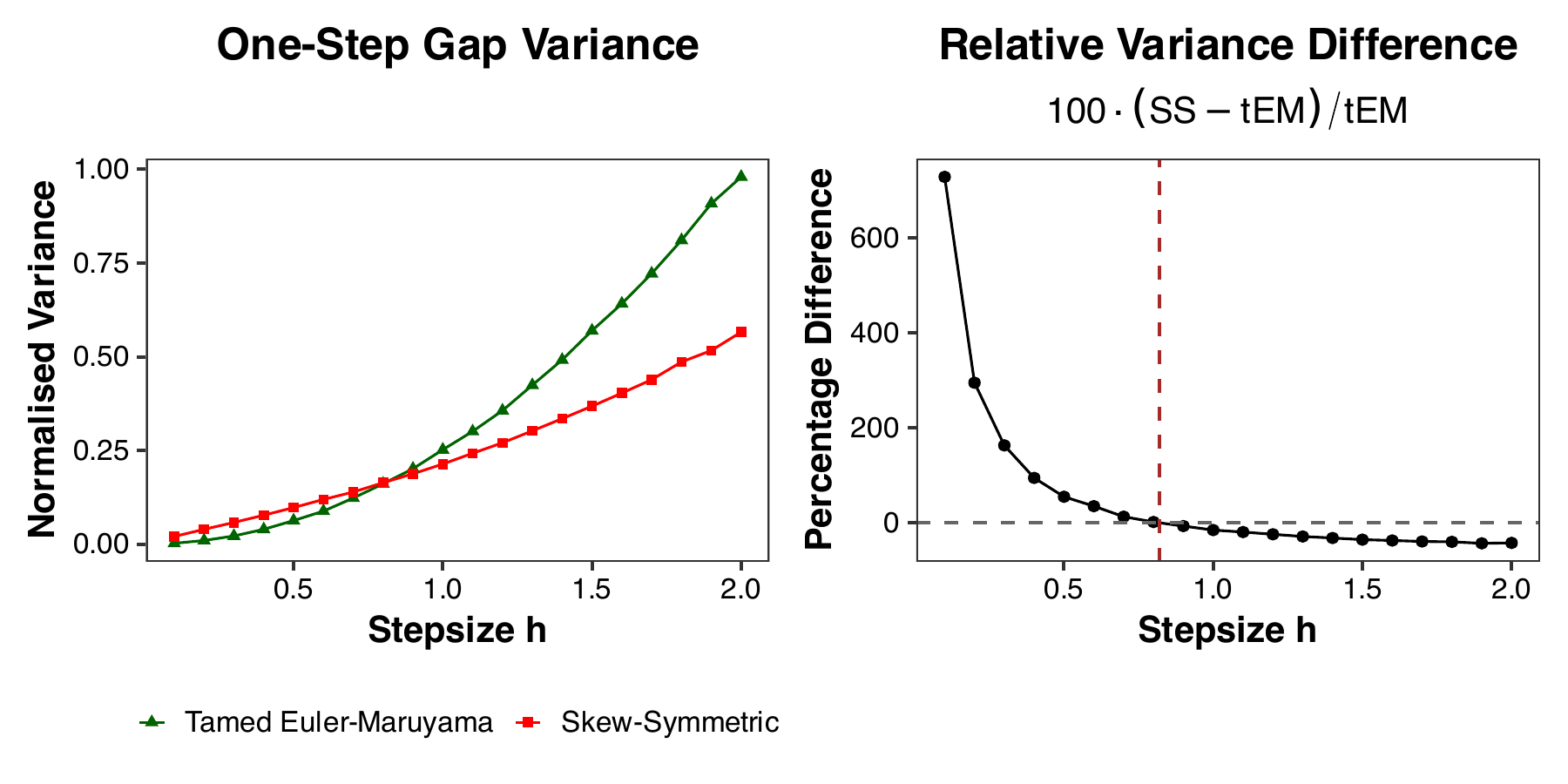}
    \caption{Normalised one-step MLMC gap variances for the skew-symmetric (red) and tamed Euler--Maruyama (green) schemes applied to the stochastic Ginzburg--Landau equation $\d X_t = (-X_t^3 + 10 X_t) \d t + X_t \d W_t$ with $X_0 = 100$. The right panel shows the relative variance difference \(100(\mathbb{V}_{\mathrm{sk}}-\mathbb{V}_{\mathrm{te}})/\mathbb{V}_{\mathrm{te}}\), illustrating the predicted crossover near \(h \approx 0.82\) of Corollary \ref{cor:comparison}.}
    \label{fig:combined_one_step_var_comparison}
\end{figure}
The main result of the section is the following.

\begin{theorem}[Large-state MLMC asymptotics]
\label{thm:main}
Fix $h>0$ and $\gamma>0$. Then, as $x \to \infty$, the following limits hold.

\begin{enumerate}[leftmargin=0.5cm]
\item[(i)] Tamed Euler.
\[
\frac{Y_{h/2}^f(x) - Y_h^c(x)}{x}
\rightarrow
a(h)^2 G_1 G_2
\quad \text{in } L^2,
\]
meaning
\[
\mathbb{V}_\mathrm{te}(h) := \lim_{x\to\infty}
\frac{1}{x^2}
\textup{Var} \bigl(Y_{h/2}^f(x) - Y_h^c(x)\bigr)
=
a(h)^4.
\]

\item[(ii)] Skew-symmetric scheme.
\[
\frac{\tilde{X}_{h/2}^f(x) - \tilde{X}_h^c(x)}{x}
\rightarrow
a(h)\Bigl(|U_1+U_2| - |U_1| - |U_2|\Bigr)
+ a(h)^2 |U_1||U_2|
\quad \text{in } L^2,
\]
meaning
\begin{align*}
& \mathbb{V}_{\mathrm{sk}}(h)
:=
\lim_{x\to\infty}
\frac{1}{x^2}\textup{Var}\bigl(\tilde{X}_{h/2}^f(x) - \tilde{X}_h^c(x)\bigr)
\\
&\quad=
\left( 2 - \frac{16}{\pi} + \frac{8\sqrt{2}}{\pi} \right) a(h)^2 
+ 
2 \left(\frac{-4 - \pi + 4\sqrt{2}}{\pi^{3/2}}\right) a(h)^3 
+ 
\left( 1 - \frac{4}{\pi^2} \right) a(h)^4.
\end{align*}
\end{enumerate}
\end{theorem}
The proof of \cref{thm:main} is provided in Appendix \ref{sec:pf_var_cpl}. As a consequence of \cref{thm:main}, we have the following comparison of the variances under the tamed Euler and skew-symmetric schemes, which we also illustrate in \cref{fig:combined_one_step_var_comparison}.
\begin{corollary} 
\label{cor:comparison}
It holds that
\[
\mathbb{V}_{\mathrm{sk}}(h) < \mathbb{V}_{\mathrm{te}}(h)
\quad \text{for } h > h_\star,
\]
and $\mathbb{V}_{\mathrm{sk}}(h) > \mathbb{V}_{\mathrm{te}}(h)$ for $0< h < h_\star$, where 
\begin{align*}
h_\star 
&= \frac{2a_\star^2}{\gamma^2} \approx \frac{0.82}{\gamma^2}, 
\\
a_\star
&:=
\frac{\pi^2}{4}
\Biggl(
\frac{-4-\pi+4\sqrt{2}}{\pi^{3/2}} + \sqrt{\left(\frac{-4-\pi+4\sqrt{2}}{\pi^{3/2}}\right)^2 + \frac{4 \bigl(2 - 16/\pi + 8\sqrt{2}/\pi \bigr)}{\pi^2}}
\Biggr) \approx 0.64.
\end{align*}
\end{corollary}
\begin{proof}
By Theorem~\ref{thm:main},
\[
\mathbb{V}_{\mathrm{sk}}(h)-\mathbb{V}_{\mathrm{te}}(h)
=
c_0 a(h)^2 + 2c_1 a(h)^3 - \frac{4}{\pi^2}a(h)^4
=
a(h)^2\Bigl(c_0 + 2c_1 a(h) - \frac{4}{\pi^2}a(h)^2\Bigr).
\]
Now $c_0>0$, $c_1<0$, and $4/\pi^2>0$. The quadratic polynomial
\[
q(a):=c_0 + 2c_1 a - \frac{4}{\pi^2}a^2
\]
is therefore strictly concave, satisfies $q(0)=c_0>0$, and obeys $q(a)\to -\infty$ as $a\to\infty$. It therefore has exactly one positive root, denoted by $a_\star:= a(h_\star)$. Solving $q(a_\star)=0$ gives
\[
a_\star
=
\frac{\pi^2}{4}
\left(
c_1 + \sqrt{c_1^2 + \frac{4c_0}{\pi^2}}
\right).
\]
Note that the sign of $\mathbb{V}_{\mathrm{sk}}(h)-\mathbb{V}_{\mathrm{te}}(h)$ is the sign of $q(a(h))$. Combining with the above implies that 
$\mathbb{V}_{\mathrm{sk}}(h)>\mathbb{V}_{\mathrm{te}}(h)$
for $0<a(h)<a_\star$,
and
$\mathbb{V}_{\mathrm{sk}}(h)<\mathbb{V}_{\mathrm{te}}(h)$
for $a(h)>a_\star$.
The final statement follows by substituting $a(h)=\gamma\sqrt{h/2}$ and the values of $c_0$ and $c_1$ given by Theorem~\ref{thm:main}.
\end{proof}

\section{Numerical experiments}
\label{sec:experiments}

In this section, we conduct experiments to investigate pathwise performance of the skew-symmetric scheme in both standard and multilevel Monte Carlo settings. 
The code used to produce the results can be found at \url{https://github.com/ShuSheng3927/skew_symmetric_pathwise}.

\subsection{Strong Convergence with the Simplified A\"it-Sahalia Model}\label{sec:numerical.strong.convergence.rate}

\begin{figure}[h!]
\centering
\includegraphics[width=\linewidth]{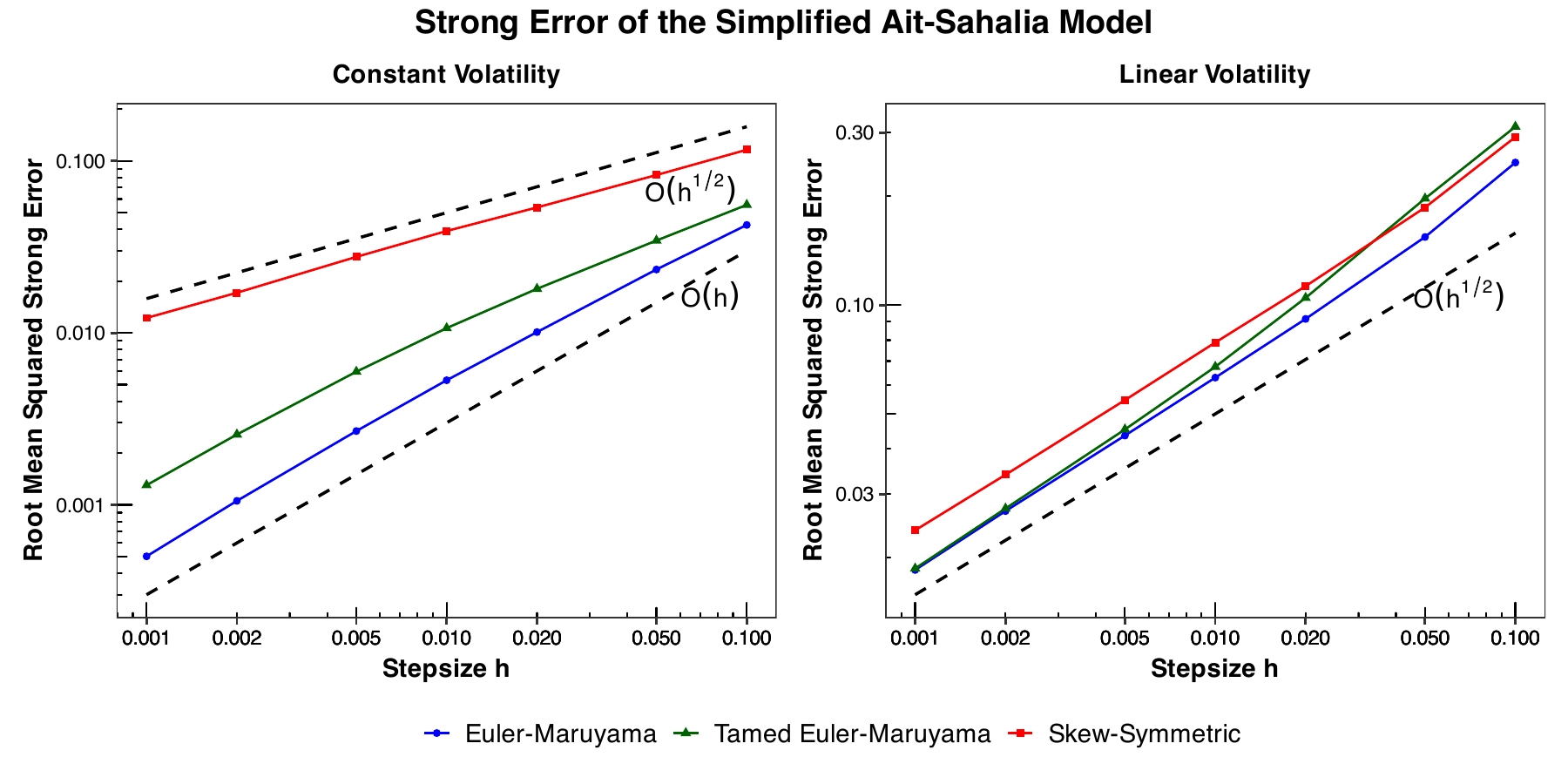}
\caption{Strong convergence of the Euler--Maruyama (blue), tamed Euler--Maruyama (green), and skew-symmetric (red) schemes for the simplified A\"it-Sahalia model with constant volatility (left) and linear volatility (right). The root-mean-square strong error is plotted against the step size $h$. The dashed black lines indicate the reference rates $O(h)$ and $O(h^{1/2})$.}
\label{fig:strong_convergence_simplified_as}
\end{figure}

To investigate the strong convergence behaviour of the considered numerical schemes, we study the simplified A\"it-Sahalia model \cite{szpruch2011numerical}
\[
\d X_t = \left( \delta + \gamma X_t - \epsilon X_t^2 \right) \d t + \beta X_t^b\, \d W_t,
\]
with initial condition $X_0=1$, terminal time $T=1$, and parameters $\delta=\gamma=\epsilon=0.5$ and $\beta=1$. We consider two choices of volatility exponent, namely $b=0$ (constant volatility) and $b=1$ (linear volatility).

Three numerical schemes are compared: Euler--Maruyama, tamed Euler, and the skew-symmetric scheme. For each method, we approximate the root-mean-square error (RMSE) at the terminal time using $M=10^5$ Monte Carlo samples and an Euler--Maruyama approximation with step size $10^{-4}$ as a reference solution. Denoting by $\widetilde{X}_h^{[m]}$ the terminal value computed using step size $h$ and by $X_{\mathrm{ref}}^{[m]}$ the corresponding reference solution, we compute
\[
\mathrm{RMSE}(h) = \left( \frac{1}{M} \sum_{m=1}^{M} \left| \widetilde{X}_h^{[m]} - X_{\mathrm{ref}}^{[m]} \right|^2 \right)^{1/2}.
\]
The approximations $\widetilde{X}_h^{[m]}$ and the reference solution $X_{\mathrm{ref}}^{[m]}$ are generated using the same underlying Brownian path for each Monte Carlo sample, corresponding to the standard synchronous coupling as described in Section \ref{sec:construction-pathwise-ss}.

The results are presented in Figure \ref{fig:strong_convergence_simplified_as}. For the linear volatility model ($b=1$), all three schemes exhibit an empirical $O(h^{1/2})$ strong convergence rate, in agreement with the rate predicted by the theory in settings where its assumptions hold. In the constant volatility setting ($b=0$), the Euler--Maruyama and tamed Euler schemes display first-order convergence, with RMSE scaling proportionally to $O(h)$. This improved rate is attributable to the additive-noise structure of the problem. The skew-symmetric scheme retains the $O(h^{1/2})$ behaviour predicted by Theorem \ref{thm:main_strong}.

\subsection{MLMC with the Stochastic Ginzburg--Landau Equation}\label{sec:numerical.mlmc}
We investigate the performance of multilevel Monte Carlo (MLMC) estimators for the stochastic Ginzburg--Landau equation \eqref{eq:SGL} with parameters $\beta= \gamma=1$ and the terminal time $T = 1$. We consider the multilevel estimator $Y = \sum_{\ell = L_0 }^L Y_\ell, \, L_0 < L,$ using the pathwise skew-symmetric (SS) \eqref{eq:ss_one_step} or the tamed Euler (TE) scheme \eqref{eq:te_one_step} with 
\begin{align*}
{Y}_{L_0}  = \frac{1}{N_{L_0}} \sum_{i = 1}^{N_{L_0}} 
f (\bar{X}_{K_{L_0}}^{[i]}), \quad 
Y_{\ell} =  \frac{1}{N_\ell} \sum_{i = 1}^{N_\ell} \Bigl[ f (\bar{X}_{K_\ell}^{[i]}) - f (\bar{X}_{K_{\ell - 1}}^{[i]})  \Bigr],  
\quad  \ell = L_0 + 1, \ldots, L, 
\end{align*} 
where $\bar{X}_{K_\ell}^{[i]}$ denotes the $i$-th sample path of the scheme TE or SS at time $T$ with the time-discretisation steps $K_\ell = 2^\ell$. The test function is set as $f(x) = x^2$. We refer to the estimator $Y$ computed by SS or TE as SS-MLMC or TE-MLMC, respectively, and compare the performance of the two MLMC estimators. 

The number of samples at each level $N_\ell$ and the maximum level $L$ are selected by following a standard procedure \cite[Section 3.1]{giles:15} and can be different for each numerical scheme following a pilot run.  In our experiment, we estimate such quantities by targeting the MLMC estimators' root-mean-square errors (RMSEs) $\varepsilon$, where $\varepsilon \in \{0.016, 0.008, 0.004, 0.002, 0.001 \}$. We also set for $\ell = L_0, L_0+1, \ldots, L$,  
\begin{align*}
N_\ell & = \left\lceil 
2 \varepsilon^{-2} \sqrt{{V_\ell}/{C_\ell}} \Biggl(\sum_{k = L_0}^L \sqrt{V_k C_k} \Biggr)   \right\rceil; \\[0.2cm] 
V_\ell & = 
\begin{cases}
    \mathrm{Var} \bigl[ f (\bar{X}^{[1]}_{K_{L_0}}) \bigr],  
    & \ell = L_0 ; \\[0.1cm] 
        \mathrm{Var} \bigl[ f (\bar{X}^{[1]}_{K_{\ell}}) - f (\bar{X}^{[1]}_{K_{\ell-1}})  \bigr], &  \ell = L_0 + 1, \ldots, L, 
\end{cases} 
\end{align*}
where $C_\ell$ is the cost of running the scheme per sample and is estimated as the median per-sample running time over five repetitions, each with 500,000 trajectories. We estimate $V_\ell$ using standard Monte Carlo (MC) with 500,000 trajectories. The finest level $L$ is then chosen so that the remaining bias, estimated from the correction means $| \mathbb{E} \bigl[ f (\bar{X}^{[1]}_{K_{\ell}}) - f (\bar{X}^{[1]}_{K_{\ell-1}}) \bigr] |$ together with the first-order weak convergence, is below $\varepsilon / \sqrt{2}$, following \cite[Section 3.1]{giles:15}.  

\begin{figure}
    \centering
    \includegraphics[width=1.0\linewidth]{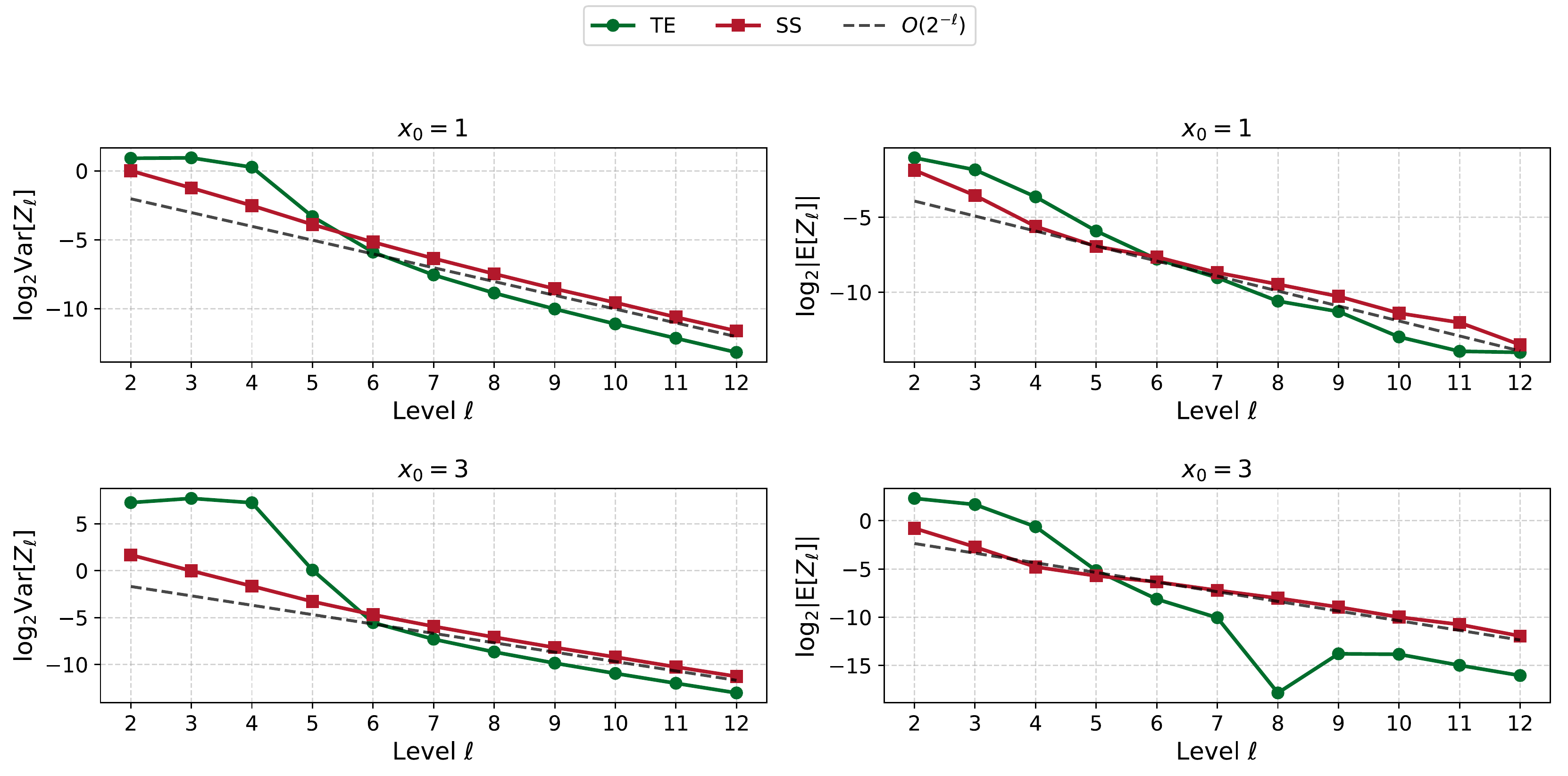}
    \caption{Variance (left) and absolute mean (right) of the MLMC corrections for the tamed Euler (TE) and the pathwise skew-symmetric (SS) schemes. The top and bottom panels correspond to $X_0 = 1$ and $X_0 = 3$. }
    \label{fig:var_mean_cpls}
\end{figure} 

\begin{figure}
    \centering
    \includegraphics[width=1.0\linewidth]{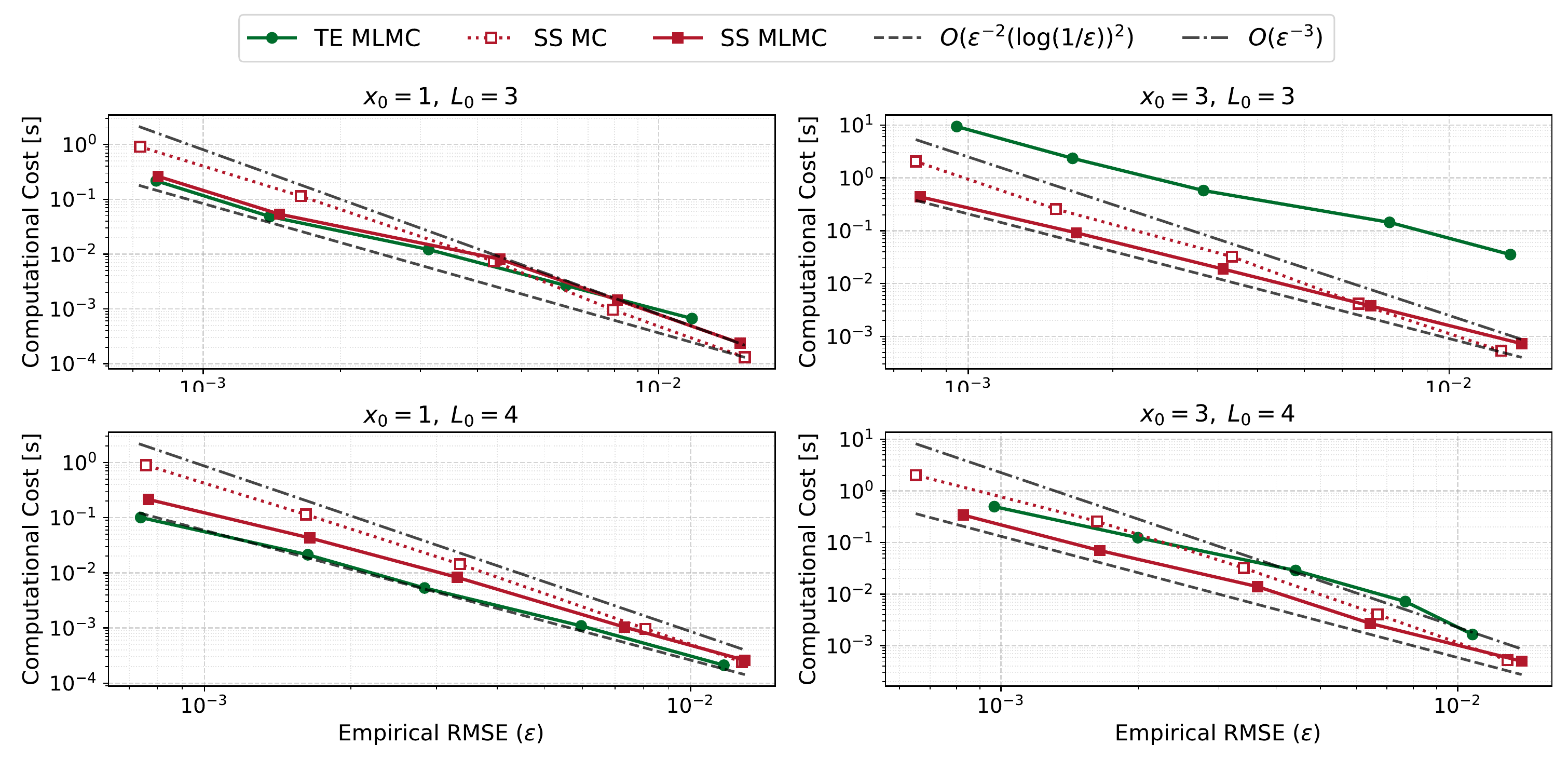}
    \caption{Computational cost versus empirical RMSE for TE-MLMC, SS-MLMC, and standard Monte Carlo based on SS (SS-MC), applied to the stochastic Ginzburg–Landau model \eqref{eq:SGL} with $\beta = \gamma = 1$. The target expectation is $\mathbb{E}_{X_0} [(X_1)^2]$. }
    \label{fig:cost_vs_rmse}
\end{figure} 

Figure \ref{fig:var_mean_cpls} plots $\mathrm{Var} \bigl[Z_\ell \bigr]$ and $| \mathbb{E} \bigl[Z_\ell \bigr] |$, $Z_\ell := f (\bar{X}^{[1]}_{K_{\ell}}) - f (\bar{X}^{[1]}_{K_{\ell-1}})$ for $\ell = 2, 3, \ldots, 12$ under the initial conditions $X_0 = 1$ and $X_0 = 3$. Notably, the SS estimator indicates moderate decays across all levels for both initial conditions, while the TE's variance is more sensitive to $X_0$ and has substantially larger values at the coarser levels, particularly for the $X_0 = 3$ case -- as expected from the directional stability analysis in Section \ref{sec:inward_drift}. After a crossover at finer levels, the TE correction variance becomes slightly smaller than that of SS, which is qualitatively consistent with the analysis concerning the linear test function in Section \ref{sec:var.multilevel}. 

Furthermore, Figure \ref{fig:cost_vs_rmse} shows the pairs of computational costs against empirical RMSEs with \{TE, SS\}-MLMC and standard MC using SS (SS-MC) under four scenarios $(X_0, L_0) \in \{ (1, 3), (1, 4), (3, 3), (3,4) \}$. In this figure, the cost of MLMC and RMSE are computed as
\begin{align*} 
(\mathrm{Cost}) = \sum_{\ell = L_0}^{L} N_\ell C_\ell, \quad 
(\mathrm{RMSE}) = \sqrt{\tfrac{1}{100} \sum_{i = 1}^{100} \bigl( Y^{[i]} - Y^\dagger \bigr)^2 },
\end{align*} 
where $Y^{[i]}$ is the $i$-th independent realisation of \{TE, SS\}- MLMC and $Y^\dagger$ denotes the reference value of $\mathbb{E}_{X_0} [(X_T)^2]$, estimated by standard MC using $5,000,000$ paths of the TE scheme with $2^{16}$ time steps. 

We observe that SS-MLMC produces reliable estimates for all scenarios and exhibits approximately $\mathcal{O}\bigl(\varepsilon^{-2} (\log (1/\varepsilon))^2 \bigr)$ cost growth -- consistent with Corollary \ref{cor:mlmc} -- whereas SS-MC's cost is approximately $\mathcal{O} (\varepsilon^{-3})$. In comparison, TE-MLMC is more sensitive to the model setting and the choice of the coarsest level $L_0$ due to the higher variance in coarse levels, though it becomes competitive with SS-MLMC under an appropriate choice of $L_0$ in the case $X_0 = 1$. Thus, the experiment demonstrates the greater robustness of SS-MLMC for the SDE with multiplicative noise \eqref{eq:SGL}.

\section{Conclusion} We have proposed a path-wise formulation of the skew-symme-\\
tric discretisation obtained by representing the skew-normal increment through a pair of independent Gaussians and coupling these with the Brownian motion driving the SDE. The resulting scheme is fully explicit and inherits the robustness of the distributional scheme of \cite{skew:ima} to superlinearly growing drift, while in addition admitting a strong error analysis. We established mean-square convergence of order $1/2$ (Theorem \ref{thm:main_strong}) and deduced that the associated multilevel estimator achieves a mean-square error $\varepsilon^2$ at cost $O\big(\varepsilon^{-2}(\log(1/\varepsilon))^2\big)$ (Corollary \ref{cor:mlmc}).

Our comparison with the tamed Euler scheme shows that the two methods differ most when the state is far from the stable region. The tamed drift increment is bounded uniformly, so the restoring force available to a single step does not grow as the state moves outwards, whereas the skew-symmetric increment contracts in proportion to the state, as the wrong direction probabilities (Proposition \ref{prop:large_state}) and the normalised mean increments (Table \ref{tab:one_step_signed_mean}) reflect. Empirically, the consequence is coarse-level coupling variances that are markedly less sensitive to the initial condition. At small step sizes, the ordering reverses (Corollary \ref{cor:comparison}). This could motivate a hybrid strategy, using the skew-symmetric scheme at coarse levels and the tamed Euler scheme at fine ones, which we leave as a direction for future work.

Two more directions seem natural, and we are currently pursuing them. The scheme as currently defined requires the uniform ellipticity of Assumption \ref{assump:vol}.I \\
through the inverse $a^{-1}$ appearing in $\alpha$ in (\ref{def:alpha}). However, many applications that motivate robust discretisation involve hypoelliptic dynamics, such as kinetic Langevin with a superquadratic potential. A second direction would be to use skew-symmetric schemes to construct splitting schemes. Part of the drift could be handled by an Ornstein--Uhlenbeck step and the remainder by skewing, with the aim of restoring some of the mobility that the skew-symmetric scheme,  similar to taming, gives up in exchange for robustness.

\subsection*{Acknowledgements} The authors would like to thank Matteo Croci and 
Michael Tretyakov for useful discussions. YI was supported by EPSRC 
(EP/Y028783/1). The authors were supported by the Heilbronn Institute for Mathematical Research and the UKRI/EPSRC Additional Funding Programme for Mathematical Sciences (EP/V521917/1).

\appendix
\section{Proof of Proposition \ref{prop:one_step_s}} \label{sec:pf_prop}

Equation \eqref{eq:local_mean} is immediate from $\mathbb{E}_x  \bigl[ \widetilde{X}_1 \bigr] = \mathbb{E}_x \bigl[ \bar{X}_1 \bigr]$ due to the coupling and then the local weak convergence of the distributional scheme \eqref{eq:ss-update} for elliptic diffusions, which is shown in \cite[Proposition 3.1]{skew:ima} under Assumptions \ref{assump:drift}, \ref{assump:vol} and \ref{assump:deriv_poly}. 

We therefore focus on proving (\ref{eq:local_Lp}). The Brownian coupling (\ref{eq:pathwise_ss}) shows that $X_h - \widetilde{X}_1 =  T_1 + T_2 + T_3$ with 
%
%
%
\begin{gather*}
T_1 := \int_0^h V_0 (X_t) \d t, \quad 
T_2 := \sum_{i = 1}^{d_W} \int_0^h \bigl( V_{i} (X_t) -  V_{i} (x) \bigr) \d W_{i, t}; \\ 
T_3 := \sum_{i = 1}^{d_W} V_{i} (x) \left\{  \left( 1 -  \frac{1}{\sqrt{1 + \tilde{\alpha}_i (h, x)^2}} \right) W_{i, h} 
- \frac{\tilde{\alpha}_i (h, x)}{\sqrt{1 + \tilde{\alpha}_i (h, x)^2}} | \widetilde{W}_{i, h} |  \right\}. 
\end{gather*} 
This implies
\begin{align*}
\mathbb{E}_x \bigl[ \bigl| X_h - \widetilde{X}_1  \bigr|^{2r} \bigr] 
\le C \sum_{k = 1}^3 \mathbb{E}_x \bigl[ | T_k |^{2r} \bigr] 
\end{align*}
for some $C = C(r) > 0$. Our goal will be to upper bound the quantities $\mathbb{E}_x \bigl[ | T_k |^{2r} \bigr]$ for each $k = 1, 2, 3$. In order to do that, we will use the fact that under Assumptions \ref{assump:drift} and \ref{assump:vol}, for any suitably large $p \ge 1$, there exists a constant $C>0$ such that for all $x \in \mathbb{R}^d$
\begin{align} \label{eq:bd_m_sde}
\sup_{t \in [0, T]}  \mathbb{E}_x [ |X_t|^{2p}] 
\le C ( 1 + |x|^{2p}),
\end{align} 
see e.g. \cite{Tret:13}. For the first term, applying Jensen's inequality gives
\begin{align*}
& \mathbb{E}_x \bigl[ | T_1 |^{2r} \bigr] 
 = h^{2r} \cdot \mathbb{E}_x \left[ \Bigl| \frac{1}{h} \int_0^h V_0  (X_t) dt \Bigr|^{2r} \right] \\
& \quad 
\le  h^{2r - 1} \int_0^h \mathbb{E}_x \bigl[ | V_0  (X_t) |^{2r} \bigr]  \d t 
\le C  h^{2r - 1} \int_0^h \mathbb{E}_x \bigl[ 1 + |X_t|^{2 r q} \bigr]  \d t
\le C(1 + |x|^{2rq}) h^{2r}, 
\nonumber 
\end{align*}
for some constants $C > 0$ and $q \ge 1$, where in the last inequality we used Assumption \ref{assump:deriv_poly} and used the moment estimate \eqref{eq:bd_m_sde} at a sufficiently large moment order. 

Subsequently, for the second term, we employ the Burkholder-Davis-Gundy inequality, Assumption \ref{assump:vol}-II and Jensen's inequality to get 
\begin{align*}
\mathbb{E}_x\bigl[|T_2|^{2r}\bigr]
&\leq
C_1\,
\mathbb{E}_x\left[
\left(
\int_0^h
\sum_{i=1}^{d_W}
|V_i(X_t)-V_i(x)|^2\,\d t
\right)^r
\right]
\\
&\leq
C_2\,
\mathbb{E}_x\left[
\left(
\int_0^h |X_t-x|^2\,\d t
\right)^r
\right]
\leq
C_2 h^{r-1}
\int_0^h
\mathbb{E}_x\bigl[|X_t-x|^{2r}\bigr]\,\d t
\\
&\leq
C_3(1+|x|^{2rq})h^{2r}.
\end{align*}
%
for some constants $C_1, C_2, C_3 > 0$ and $q \ge 1$, where in the last line we have used \eqref{eq:bd_m_sde} with a sufficiently large moment order, together with the inequality
\begin{align} \label{eq:bd_diff_X}
\mathbb{E}_x \bigl[ | X_t - x |^{2r} \bigr] \le C ( 1 + |x|^{2rq}) t^{r},  
\end{align}  
which holds under the stated conditions. 
%
For the third term, we have that 
\begin{align}
\mathbb{E}_x \bigl[ |T_3|^{2r} \bigr] 
& \le C \times \bigl( \mathcal{E}^I + \mathcal{E}^{II} \bigr)  
\end{align}
for some $C = C(r) > 0 $ with 
\begin{align}
\mathcal{E}^I & := \mathbb{E} \left[ 
\left| \sum_{i = 1}^{d_W} 
V_{i} (x)  \left( 1 - \frac{1}{\sqrt{1 + \tilde{\alpha}_i (h, x)^2}} \right) W_{i, h} \right|^{2r} \right],  \\ 
\mathcal{E}^{II} & := \mathbb{E} \left[ 
\left| \sum_{i = 1}^{d_W} 
V_{i} (x)  \left( \frac{\tilde{\alpha}_i (h, x)}{\sqrt{1 + \tilde{\alpha}_i (h, x)^2}} \right) |\widetilde{W}_{i, h}| \right|^{2r} \right]. 
\end{align} 
For the term $\mathcal{E}^I$, we recall the definition of $\tilde{\alpha}_{i} (h,x)$ in (\ref{eq:skewness}) and apply the Rosenthal inequality (e.g. \cite{Rose:70} or \cite[Lemma A.2.]{Soni:25}) to get 
\begin{align*}
\mathcal{E}^I 
& \le C_1   \sum_{i = 1}^{d_W} \mathbb{E} \left[
\Bigl| V_{i} (x)  \Bigl( 1 - \frac{1}{\sqrt{1 + \tilde{\alpha}_i (h, x)^2}} \Bigr) W_{i, h}  \Bigr|^{2r} \right] \\ 
& \quad  + C_1 \Biggl( \sum_{i = 1}^{d_W} 
\mathbb{E} \left[
\Bigl| V_{i} (x)  \Bigl( 1 - \frac{1}{\sqrt{1 + \tilde{\alpha}_i (h, x)^2}} \Bigr) W_{i, h}  \Bigr|^{2} \right] \Biggr)^{r} 
\le C_2 (1 + |x|^{2qr}) h^{3r}, 
\end{align*}
for some constants $C_1, C_2 > 0$ and $q \ge 1$, where we have used that 
\begin{align}
1 - \frac{1}{\sqrt{1 + \tilde{\alpha}_i (h, x)^2}} =  \frac{\sqrt{1 + \tilde{\alpha}_i (h, x)^2}- 1}{\sqrt{1 + \tilde{\alpha}_i (h, x)^2}} \le \frac{\tilde{\alpha}_i (h, x)^2}{2}
\le C ( 1 + |x|^m) h,
\end{align}
for some $m \ge 1$, which follows from Assumptions \ref{assump:drift}, \ref{assump:vol}, \ref{assump:deriv_poly} and the inequality $\sqrt{1 + \xi} - 1 \le \tfrac{\xi}{2}$ for $\xi \ge 0$. Finally, for the term $\mathcal{E}^{II}$, we express $\widetilde{W}_{i,h}$ as $\sqrt{h} Z_i, \, i = 1,\ldots d_W$ with $Z = (Z_1, \ldots, Z_{d_W}) \sim N(0,I_{d_W})$ from which it can be deduced that
\begin{align*}
\mathcal{E}^{II} & \le 
C 
\Bigl| \sum_{i = 1}^{d_W} 
V_{i} (x)  \Bigl( \tfrac{\tilde{\alpha}_i (h, x)}{\sqrt{1 + \tilde{\alpha}_i (h, x)^2}} \Bigr) \sqrt{h} \mathbb{E} |Z_i| \Bigr|^{2r} \\
& \quad + C \mathbb{E} \Bigl[ 
\Bigl| \sum_{i = 1}^{d_W} 
V_{i} (x)  \Bigl( \tfrac{\tilde{\alpha}_i (h, x)}{\sqrt{1 + \tilde{\alpha}_i (h, x)^2}} \Bigr) \sqrt{h} (|Z_i| - \mathbb{E} |Z_i| )\Bigr|^{2r} \Bigr] =: \mathcal{E}^{II}_1 + \mathcal{E}_2^{II} 
\end{align*} 
for some $C = C(r) > 0$. It then follows from (\ref{eq:skewness}) that under Assumptions \ref{assump:drift}, \ref{assump:vol} and \ref{assump:deriv_poly}, 
\begin{align} \label{eq:E_II_1}
\mathcal{E}_1^{II} \le C (1 + |x|^{2rq}) h^{2r}, 
\end{align}
for some constants $C > 0, \, q \ge 1$. For the term $\mathcal{E}_2^{II}$, we apply the Rosenthal inequality 
to obtain 
\begin{align}
\mathcal{E}_2^{II} 
& \le C \sum_{i = 1}^{d_W} \mathbb{E} \left[ \Bigl| V_i (x)  \Bigl( \tfrac{\tilde{\alpha}_i (h, x)}{\sqrt{1 + \tilde{\alpha}_i (h, x)^2}} \Bigr)  \sqrt{h} (|Z_i| - \mathbb{E} |Z_i| ) \Bigr|^{2r} \right] \\
& \qquad + C 
\Bigl( \sum_{i = 1}^{d_W} \mathbb{E} \left[ \Bigl| V_{i} (x)  \Bigl( \tfrac{\tilde{\alpha}_i (h, x)}{\sqrt{1 + \tilde{\alpha}_i (h, x)^2}} \Bigr) \sqrt{h} (|Z_i| - \mathbb{E} |Z_i| ) \Bigr|^{2} \right] 
\Bigr)^r,   
\end{align} 
for $C > 0$. 
We then see that $\mathcal{E}_2^{II} $ is bounded as the right hand side of (\ref{eq:E_II_1}) by noticing that $\mathbb{E} \Bigl[ \bigl|  |Z_i| - \mathbb{E} |Z_i| \bigr|^{2k} \Bigr]$ is finite for all $k \in \mathbb{N}$ due to the sub-Gaussianity of $|Z_i| - \mathbb{E} |Z_i|$ (e.g. Theorems 2.6 \& 2.26 in \cite{Wain:19}). We therefore complete the proof of (\ref{eq:local_Lp}) and conclude \cref{prop:one_step_s}.  $\Box$

\section{Proof of Proposition \ref{prop:large_state}} \label{sec:pf_large_state}

We start with the case of the tamed Euler scheme. We notice that the events $\{ B_i (\varepsilon, x) \}_i$ are independent because each component of the scheme is driven by independent Gaussian variables, meaning
\begin{align*}
\mathbb{P} \Bigl( \bigcup_{i \in \Lambda} B_i (\varepsilon, x)  \Bigr) & = 1 - \prod_{i \in \Lambda} \mathbb{P} \Bigl(\bigl(B_i (\varepsilon, x)\bigr)^c  \Bigr) \\ 
& = 1 - \prod_{i \in \Lambda} \Phi \Bigl( \tfrac{1}{\sqrt{h} |\sigma_i (x)|} \bigl( \varepsilon - \tfrac{h b_i(x)}{1 + h | b_i (x)|} \mathrm{sign} (x_i) \bigr) \Bigr),  
\end{align*} 
where $z \mapsto \Phi (z)$ is the CDF of a $N(0,1)$ random variable. 
If $\lim_{\rho_\Lambda(x) \to \infty}|\sigma_i(x)| = \widetilde{\sigma}_i \in (0, \infty]$, then the limit \eqref{eq:te_limit} is immediate from the above equation. 
%
%

Subsequently, we consider the skew-symmetric scheme. Again, the independence of the events $\{A_i (\varepsilon, x)\}_i$ gives 
\begin{align*}
& \mathbb{P} \Bigl( \bigcup_{i \in \Lambda} A_i (\varepsilon, x) \Bigr) \\ 
& \quad = 1 - \prod_{i \in \Lambda} \mathbb{E}_{\eta_i} 
\Bigl[ \Phi \Bigl(\tfrac{\varepsilon}{\sqrt{h} |\sigma_i (x)|} \sqrt{1 + \tilde{\alpha}_i (h, x)^2} - \sqrt{\tfrac{\pi}{2}} \tfrac{b_i (x)}{|\sigma_i (x)|} \sqrt{h} | \eta_i | \mathrm{sign} (x_i) \Bigr) \Bigr]. 
\end{align*}
Because of the boundedness of the normal CDF, the dominated convergence theorem allows the exchange of the limit and the expectation, leading to the conclusion that for any fixed $d$ and non-empty set $\Lambda \subseteq \{1, \ldots, d\}$,  
\begin{align}
\lim_{\rho_\Lambda (x)  \to \infty} 
\mathbb{P} \Bigl( \bigcup_{i \in \Lambda} A_i (\varepsilon, x) \Bigr) = 1 - \prod_{i \in \Lambda} \mathbb{E}_{\eta_i} \Bigl[ \Phi (\infty)  \Bigr] = 0, 
\end{align} 
which completes the proof. $\Box$
\section{Proof of \cref{thm:main}} \label{sec:pf_var_cpl}
(i) Set $d_\eta(y) := \eta b(y)/(1+\eta |b(y)|)$ to be the mean increment of the tamed Euler scheme.  Note that $|d_\eta(y)| \le 1$ for all $y$, meaning $d_\eta(y)/x \to 0$ as $x \to \infty$ .  The coarse update therefore satisfies
\[
\frac{Y_h^c(x)}{x}
=
1 + \frac{d_h(x)}{x} + \gamma \sqrt{h}G_c \rightarrow
1 + a(h)(G_1 + G_2)
\]
in $L^2$.  Next, for the first fine half-step,
\[
\frac{Y_1(x)}{x}
:=
\frac{T_{h/2}(x;G_1)}{x}
=
1 + \frac{d_{h/2}(x)}{x} + a(h) G_1
\rightarrow
1 + a(h) G_1
\]
in $L^2$.  The second half-step can be written
\[
\frac{Y_{h/2}^f(x)}{x}
=
\frac{Y_1(x)}{x}
+
\frac{d_{h/2}(Y_1(x))}{x}
+
a(h)\frac{Y_1(x)}{x}\,G_2.
\]
Again, $|d_{h/2}(Y_1(x))|/x \le 1/x \to 0$, and since $Y_1(x)/x \to 1+a(h)G_1$ in $L^2$
\[
\frac{Y_{h/2}^f(x)}{x}
\rightarrow
1 + a(h) G_1 + a(h)(1 + a(h) G_1) G_2
=
1 + a(h)(G_1 + G_2) + a(h)^2 G_1 G_2
\]
in $L^2$.  Subtracting the coarse limit gives
\[
\frac{Y_{h/2}^f(x) - Y_h^c(x)}{x}
\rightarrow
a(h)^2 G_1 G_2.
\]
in $L^2$.  Since $G_1G_2$ has variance $1$, the variance limit follows.  For (ii), observe that as $x \to\infty$
\[
\frac{b(x)}{\gamma x} = \frac{-x^2 + \beta}{\gamma} \to -\infty,
\]
so for each fixed $\eta>0$, $\tilde{\alpha}(\eta, x)\to -\infty$, $\beta_\eta(x)\to 0$, and $\delta_\eta(x)\to -1$ as $x \to \infty$.  Define
\[
M_x := \frac{S_{h/2}(x;U_1,V_1)}{x}
=
1 + a(h)\bigl(\beta_{h/2}(x)V_1 + \delta_{h/2}(x)|U_1|\bigr).
\]
By the boundedness of $\beta_{h/2}$ and $\delta_{h/2}$, the finiteness of all Gaussian moments, and dominated convergence,
$M_x \to 1-a(h)|U_1|$ in $L^4$.
Next, note that
\[
\frac{\tilde{X}_{h/2}^f(x)}{x}
=
M_x
\Bigl[
1 + a(h)\bigl(
\beta_{h/2}(xM_x)V_2
+
\delta_{h/2}(xM_x)|U_2|
\bigr)
\Bigr].
\]
Since $|M_x| \longrightarrow |1-a(h)|U_1||$ in probability, and the limiting random variable is non-zero almost surely, we have
$
|xM_x| \to \infty
$
in probability, meaning
$\beta_{h/2}(xM_x)\to 0$ and $\delta_{h/2}(xM_x)\to -1$
in probability. Since these coefficients are uniformly bounded,
\[
|\beta_{h/2}(xM_x)V_2|^4\le |V_2|^4,
\]
and
\[
\bigl|(\delta_{h/2}(xM_x)+1)|U_2|\bigr|^4
\le 16|U_2|^4.
\]
Both dominating random variables are integrable. By convergence in probability together with dominated convergence, therefore,
\[
\beta_{h/2}(xM_x)V_2\to0,
\qquad
\delta_{h/2}(xM_x)|U_2|\to-|U_2|
\]
in $L^4$. Applying H\"older's inequality then gives
\[
\frac{\tilde{X}_{h/2}^f(x)}{x}
\longrightarrow
(1-a(h)|U_1|)(1-a(h)|U_2|)
\]
in $L^2$.  For the coarse step,
\[
\frac{\tilde{X}_h^c(x)}{x}
=
1 + \gamma \sqrt{h}\bigl(\beta_h(x)V_c + \delta_h(x)|U_c|\bigr)
\rightarrow
1 - a(h)|U_1+U_2|
\]
in $L^2$.  Subtracting gives
\[
\frac{\tilde{X}_{h/2}^f(x) - \tilde{X}_h^c(x)}{x}
\rightarrow
(1 - a(h)|U_1|)(1 - a(h)|U_2|) - \bigl(1 - a(h)|U_1+U_2|\bigr),
\]
which simplifies to the first result of (ii). For the variance, let
\[
Z := |U_1+U_2| - |U_1| - |U_2|,
\quad
W := |U_1||U_2|.
\]
Then 
\[
\mathbb{V}_{\mathrm{sk}}(h)
=
\text{Var}(a(h)Z + a(h)^2 W)
=
a(h)^2 \text{Var}(Z) + 2 a(h)^3 \text{Cov}(Z,W) + a(h)^4 \text{Var}(W).
\]
Direct calculation gives
\[
\text{Var}(Z) = 2 - \frac{16}{\pi} + \frac{8\sqrt{2}}{\pi},
\quad
\text{Cov}(Z,W) = \frac{-4 - \pi + 4\sqrt{2}}{\pi^{3/2}},
\quad
\text{Var}(W) = 1 - \frac{4}{\pi^2}.
\]
Substituting these identities completes the proof. $\Box$

\bibliographystyle{plain}
\bibliography{references}

\begin{thebibliography}{10}

\bibitem{azzalini2013skew}
Adelchi Azzalini.
\newblock {\em {The Skew-Normal and Related Families}}, volume~3.
\newblock Cambridge University Press, 2013.

\bibitem{Fang:20}
Wei Fang and Michael~B. Giles.
\newblock {Adaptive {E}uler–{M}aruyama method for {SDE}s with nonglobally {L}ipschitz drift}.
\newblock {\em {The Annals of Applied Probability}}, 30(2), 2020.

\bibitem{Giles:08}
Michael~B. Giles.
\newblock {Multilevel {M}onte {C}arlo path simulation}.
\newblock {\em {Operations Research}}, 56(3):607–617, 2008.

\bibitem{giles:15}
Michael~B. Giles.
\newblock {M}ultilevel {M}onte {C}arlo methods.
\newblock {\em Acta Numerica}, 24:259–328, April 2015.

\bibitem{Hen:86}
Norbert Henze.
\newblock {A probabilistic representation of the `skew-normal' distribution}.
\newblock {\em {Scandinavian Journal of Statistics}}, pages 271--275, 1986.

\bibitem{Hutzenthaler.Jentzen.Kloeden:2011}
Martin Hutzenthaler, Arnulf Jentzen, and Peter~E. Kloeden.
\newblock {Strong and weak divergence in finite time of {Euler's} method for stochastic differential equations with non-globally {Lipschitz} continuous coefficients}.
\newblock {\em { {{Proceedings of the Royal Society A: Mathematical, Physical and Engineering Sciences}} }}, 467(2130):1563--1576, 2011.

\bibitem{hutzenthaler.jentzen.lloeden:12}
Martin Hutzenthaler, Arnulf Jentzen, and Peter~E. Kloeden.
\newblock {Strong convergence of an explicit numerical method for {SDEs} with non-globally {Lipschitz} continuous coefficients}.
\newblock {\em {The Annals of Applied Probability}}, 22(4):1611--1641, 2012.

\bibitem{Hut:13}
Martin Hutzenthaler, Arnulf Jentzen, and Peter~E Kloeden.
\newblock {Divergence of the multilevel {M}onte {C}arlo {E}uler method for nonlinear stochastic differential equations}.
\newblock {\em {The Annals of Applied Probability}}, 23(5):1913--1966, 2013.

\bibitem{skew:ima}
Yuga Iguchi, Samuel Livingstone, Nikolas N{\"u}sken, Giorgos Vasdekis, and Rui-Yang Zhang.
\newblock {Skew-symmetric schemes for stochastic differential equations with non-{L}ipschitz drift: an unadjusted {B}arker algorithm}.
\newblock {\em {IMA Journal of Numerical Analysis}}, 2026.

\bibitem{karatzas2014brownian}
Ioannis Karatzas and Steven Shreve.
\newblock {\em {Brownian Motion and Stochastic Calculus}}.
\newblock Springer, 2014.

\bibitem{krauth2006statistical}
Werner Krauth.
\newblock {\em {Statistical Mechanics: Algorithms and Computations}}, volume~13.
\newblock OUP Oxford, 2006.

\bibitem{livingstone.zanella:22}
Samuel Livingstone and Giacomo Zanella.
\newblock {The {Barker} proposal: Combining robustness and efficiency in gradient-based {MCMC}}.
\newblock {\em { {{Journal of the Royal Statistical Society Series B: Statistical Methodology}} }}, 84(2):496--523, 2022.

\bibitem{mao:15}
Xuerong Mao.
\newblock { {{The truncated {Euler}--{Maruyama} method for stochastic differential equations}} }.
\newblock {\em {Journal of Computational and Applied Mathematics}}, 290:370--384, 2015.

\bibitem{Roberts.Tweedie:96}
Gareth~O. Roberts and Richard~L. Tweedie.
\newblock {Exponential convergence of {Langevin} distributions and their discrete approximations}.
\newblock {\em {Bernoulli}}, 2(4):341--363, 1996.

\bibitem{Rose:70}
Haskell~P Rosenthal.
\newblock {On the subspaces of ${L}_p (p > 2)$ spanned by sequences of independent random variables}.
\newblock {\em {Israel Journal of Mathematics}}, 8(3):273--303, 1970.

\bibitem{sabanis:13}
Sotirios Sabanis.
\newblock {A note on tamed {Euler} approximations}.
\newblock {\em {Electronic Communications in Probability}}, 18:1--10, 2013.

\bibitem{Soni:25}
Simran Soni, Chaman Kumar, Goncalo~dos Reis, et~al.
\newblock {Tamed {E}uler approximation for fully superlinear growth {M}c{K}ean-{V}lasov {SDE} and their particle systems: sharp rates for strong propagation of chaos, convergence and ergodicity}.
\newblock {\em {arXiv preprint arXiv:2510.16427}}, 2025.

\bibitem{szpruch2011numerical}
Lukasz Szpruch, Xuerong Mao, Desmond~J Higham, and Jiazhu Pan.
\newblock {Numerical simulation of a strongly nonlinear Ait-Sahalia-type interest rate model}.
\newblock {\em {BIT Numerical Mathematics}}, 51(2):405--425, 2011.

\bibitem{Tret:13}
M.~V. Tretyakov and Z.~Zhang.
\newblock {A fundamental mean-square convergence theorem for {SDE}s with locally {L}ipschitz coefficients and its applications}.
\newblock {\em {SIAM Journal on Numerical Analysis}}, 51(6):3135–3162, 2013.

\bibitem{Wain:19}
Martin~J. Wainwright.
\newblock {\em {High-{D}imensional {S}tatistics: {A} {N}on-{A}symptotic {V}iewpoint}}.
\newblock Cambridge University Press, 2019.

\end{thebibliography}

\end{document}